\pdfoutput=1
\documentclass[reqno]{amsart} 
\usepackage{amsmath,amsthm,amsfonts,amssymb}
\usepackage{paralist}
\usepackage{graphicx}
\usepackage{verbatim}
\usepackage{epsfig}
\usepackage{epstopdf}
\usepackage{color}
\usepackage{subcaption}

\theoremstyle{plain}
\newtheorem{theorem}{Theorem}[section]
\newtheorem{lemma}{Lemma}[section]
\newtheorem{corollary}{Corollary}[section]
\newtheorem{proposition}{Proposition}[section]

\theoremstyle{definition}
\newtheorem{definition}{Definition}[section]
\newtheorem{assumption}{Assumption}[section]

\theoremstyle{remark}
\newtheorem{remark}{Remark}[section]
\newtheorem{example}{Example}[section]

\numberwithin{equation}{section}

\title[Shadow wave tracking and pressureless gas]
{Shadow wave tracking procedure and initial data 
problem for pressureless gas model}
\author{Marko Nedeljkov and Sanja Ru\v{z}i\v{c}i\'{c}}
\address{Department of Mathematics and Informatics, University of Novi Sad,
Trg D.\ Obradovi\'{c}a 4, 21000 Novi Sad, Serbia}
\email{marko@dmi.uns.ac.rs, sanja.ruzicic@dmi.uns.ac.rs}


\begin{document}

\begin{abstract} In this paper the new procedure for a construction of
an approximated solution to initial data problem for one-dimensional
pressureless gas dynamics system is introduced. The procedure is based on
solving the Riemann problems and tracking singular wave interactions.
For that system the new problem with initial data containing
Dirac delta function is solved whenever two waves interact. Use of the shadow waves as singular
solutions to such problems enables us to easily solve the interaction
problems. That permits us to make a simple extension of the
well known Wave Front Tracking algorithm. A non-standard part of the new
algorithm is dealing with delta functions as a part of a solution. In the
final part of the paper we show that the approximated solution has
a subsequence converging to a signed Radon measure. 
\end{abstract}

\maketitle

\section{Introduction}

In the last few decades, a lot of conservation law systems with 
non-classical, unbounded weak solutions were analyzed. 
One can find a lot of examples in the references at the end of the paper. 
Almost all these solutions contain the Dirac delta function that is 
not suitable for nonlinear operations. That is a source of big
problems in solving some conservation law systems. 
There are several methods for dealing with that, and 
some of them can be found in the references below. Riemann problem 
is almost fully understood for these systems, so a natural next step is
to look for a solution to a general initial data problem. Because 
of that we will use shadow waves defined in \cite{mn2010}. 
Shadow wave solutions (SDW) are represented by nets of
piecewise constant functions with respect to the time variable 
depending on a small parameter $\varepsilon>0$ tending to zero. 
A shadow wave approximates a significant number
of different types of singular solutions that differ from classical
solutions by containing the Dirac delta function supported by a shock curve. 
Their use permits one to easily find a solution 
to the interaction problem and that will be 
of the greatest importance for the construction of a solution here.
To demonstrate these ideas, we will use the well known 
pressureless gas dynamics system 
\begin{equation}\label{chRS}
\begin{split}
\partial_t \rho +\partial_x(\rho u) & = 0\\
\partial_t(\rho u) + \partial_x (\rho u^2) & = 0
\end{split}
\end{equation}
that describes an evolution of density $\rho\geq 0$ and velocity $u$ of a
fluid. The equations in (\ref{chRS}) express conservation of mass and
linear momentum in an absence of pressure. That means that changes in
internal energy manifested through temperature or specific entropy are
neglected. The above system is sometimes called the sticky particle model.
That name comes from the fact that colliding particles fuse into a single
particle that combines their masses and moves with a velocity that conserves
the total linear momentum (see \cite{b_g} or \cite{CD} for example). For
example, it models one-dimensional isentropic flow in the Eulerian description
of a thermoelastic fluid in a duct. System (\ref{chRS}) is weakly
hyperbolic with the double eigenvalue $\lambda_i(\rho,u)=u$, $i=1,2$
with both fields being linearly degenerate. It allows a mass concentration
that leads to singular, unbounded solutions containing 
the Dirac delta function. 
The system attracts great attention in the literature. Riemann
problems for the pressureless gas dynamics system with a source are analyzed
in \cite{d_mn, Shen}, two--dimensional case can be found in \cite{SZ}, while
the system with added energy conservation law is investigated in
\cite{mn2010}. Besides it, there are a significant number of
conservation laws admitting unbounded solutions. More about their origin
and history one can find in \cite{LF, Keyfitz, mn2004}. Unbounded
solutions for weakly hyperbolic systems like (\ref{chRS}) were firstly found
and they are called delta shocks. Some other interesting solutions
called singular shocks appearing in some strictly hyperbolic systems (\cite{KK}), 
or in chromatography system that changes type (\cite{Mazz, meina-ch}). It is
known that a Riemann problem for (\ref{chRS}) with the left and right
initial states $(\rho_l,u_l)$ and $(\rho_r,u_r)$ has a self-similar,
classical entropy solution that consists of two contact discontinuities
connected with the vacuum state if $u_l<u_r$, or 
a single contact discontinuity if $u_l=u_r$. If $u_l>u_r$, there
exists a non-classical solution containing the delta function. 

The authors in \cite{WRS} constructed a global weak solution to the initial
data problem for (\ref{chRS}) by using generalized variational method.
Almost at the same time, the existence of a weak solution 
to the same problem was proved in \cite{b_g}. 
Uniqueness is proved in \cite{HW} for initial data
belonging to the space of Radon measures by using methods from \cite{WRS}. 
In \cite{boudin}, the author proved existence of a solution
to classical initial data problem for (\ref{chRS}) by
using viscosity approximation. The solution is understood in the sense of
duality that is defined in \cite{BJ}. Global existence of a 
measure--theoretic solution where $\rho$ belongs to 
Borel measures space and $u$ is
square integrable with respect to $\rho$ was proved in \cite{C_S_W} by
using the theory of first-order differential inclusions in the space of
monotone transport maps introduced in \cite{Natile_Savare}. The authors in
\cite{Nguyen_Tudorascu} were using the usual entropy solution to a scalar
conservation law to obtain a global solution, while initial data
could contain a Borel measure.  Methods used in all the papers cited above 
are specific for the pressureless gas (sticky particles) model. 
Our idea is to use a procedure of shadow wave tracking because 
it can be adapted to some other system possessing unbounded
solutions. Model (\ref{chRS}) should be understood as a starting 
point for using this method in a general case. 
The logical and straightforward generalization is $3\times 3$
pressureless gas dynamic system 
\begin{equation} \label{PGD}
\begin{split}
\partial_{t}\rho+\partial_{x}(\rho u) &= 0 \\
\partial_{t}(\rho u)+\partial_{x}(\rho u^{2}) &= 0 \\
\partial_{t}(\rho u^{2}/2+\rho e)+\partial_{x}((\rho u^{2}/2+\rho e)u)
& = 0
\end{split}
\end{equation}
described in \cite{mn2010}. That system has the similar structure as 
(\ref{chRS}), and we will note small changes in the procedure.

The main idea for the approximate solution construction procedure comes from
the well known Wave Front Tracking (WFT) algorithm 
(see \cite{AB, ABWFT, H_R,Risebro}). The procedure starts with 
an approximation of initial data by piecewise constant 
functions and tracking the waves and monitoring their interactions 
later on.  The shadow waves are approximations of delta
shock solution and due to their construction one can use an algorithm
similar to the WFT one. One of the main difficulties in 
the WFT algorithm was the fact that the number of wave fronts may 
approach infinity within the finite time for $n\times n$ systems, with $n>2$. 
Here, we are dealing with $2\times 2$ system in which this problem 
does not occur. In this particular case, a number of waves decreases
after each interaction as one can see below. But, the resulting wave front 
here is not necessarily a straight line (i.e.\ the wave propagates 
with a non--constant speed), which is not a case with WFT algorithm
for BV solutions.  That is a consequence of the fact that a
shadow wave interaction with some wave produces a new shadow wave with
non-constant speed in general.  So, we have to deal with the additional
problem of analyzing such wave front curves.

As we already mentioned, the procedure for finding an approximate solution to
the initial data problem presented in this paper can be used for general 1D
conservation law systems. It is only required that they admit 
a unique solution to the corresponding Riemann problem 
consisting of elementary and shadow waves combinations. 
That is the first advantage of our solution construction compared to the
methods previously discussed that depend on a particular form of
conservation law system. There are some peculiarities 
in the pressureless gas model. The absence of rarefaction waves 
makes the procedure simpler. But, on the other side, the
appearance of vacuum in the approximate solution was the main source of
difficulties in the approximate solution construction.
Also, that makes a limiting process harder to follow since there are 
no vacuum areas in a local smooth solution to the system. 
The ultimate step would be to generalize the procedure for (\ref{chRS}) 
to obtain a general algorithm for solving a wide class of 
conservation law systems admitting unbounded solutions. 
Note that there is an example of shadow wave interactions that  
cannot be handled in the way used here, as proved in  
\cite{mn2014} for the model of Chaplygin gas.

The second advantage of the procedure is that it can be adapted for a
numerical implementation. A complete verification is left for future
research since the procedure in the paper requires some additional work 
to obtain relevant numerical results. For example, one has to develop
an efficient procedure that will provide a good approximation for the next
interaction point, especially when an interaction order between waves is
not known in advance. The use of the exact values demands a
huge computation effort and one cannot control an approximation error.

The first main result in this paper is the construction of a global admissible
approximate solution to the initial data problem for (\ref{chRS}). The
initial data are bounded piecewise $\mathcal{C}^{1}$ functions with a finite
number of jumps. The second one is the existence of subsequence 
converging in the space of signed Radon measures.
Moreover, there exists a subsequence converging to
a measure that consists of classical solutions connected by delta
function at least for a small time interval.	
In that time interval, the approximate solution can be obtained
uniquely using a kind of well--balanced partitions.

Note that the Lax entropy condition (a convex entropy--entropy flux pair) 
does not suffices to single out all non-physical solutions for (\ref{chRS}) 
as proved in \cite{WRS}. 
One has to use overcompressibility to extract a proper solution.
It means that all characteristics run into a shock front 
(especially, $u_l>u_r$ for system (\ref{chRS})).
Concerning other systems admitting singular solutions, there are some
interesting facts about relations between these two admissibility conditions.
As it was shown in \cite{mn2010} for (\ref{chRS}) with
the energy conservation equation added, they  
are equivalent for all semi--convex entropies $\eta$.
But the overcompressibility condition can be 
weaker as shown in \cite{mn_sr}. When dealing
with isentropic gas dynamics systems, the authors often use the energy
inequality, derived from energy conservation law as an additional criterion
for admissibility check (see \cite{dl_s} for example).
The energy density for pressureless gas is $E=\frac{1}{2}\rho u^2$.
Here we present a
simple analysis of energy propagation but we did not use it for choosing a
proper solution.

The paper is organized as follows. Section \ref{sec:statement} contains a
statement of the problem as well as an overview of all waves which appear as
a part of a solution to the Riemann problem. Section
\ref{sec:interactions} is devoted to an analysis of shadow wave interactions.
We describe all interactions between two or more waves which may occur at some
time in Section \ref{sec:algorithm}.
After that, a detailed presentation of procedure which provides
a scheme for constructing the admissible approximate 
solution to the initial value problem is given. The procedure 
is based on the approximation of initial data
and tracking interactions between the waves which are 
obtained as solutions to the Riemann problems. 
A relation between each pair of consecutive states obtained by the
initial data approximation contains all information needed for 
the construction of a solution after each interaction point. 
Details depend on monotonicity of the initial functions $u(x)$ and $\rho(x)$. 
Section \ref{sec:analysis} contains proofs of admissible 
approximated solution existence to the initial data problem when 
the function $u(x)$ is monotone. That result is then extended 
for $u(x)$ having a finite number of extremes.
In Section \ref{sec:energy} we briefly discuss entropy changes across a shadow
wave and after the interactions and we prove that the total entropy decreases after the interaction between two shadow waves. 
The remainder of the paper is devoted to
proving that solution converges in the space of measures and that a limit is unique in some sense and at least for some time.

\section{Riemann problems}\label{sec:statement}

In the rest of this paper we will write
$a_\varepsilon\sim b_\varepsilon$ if there exists $A>0$
such that $\lim_{\varepsilon\to 0} \frac{a_\varepsilon}{b_\varepsilon}=A$.
The sign ``$\approx$'' will denote the distributional limit
as $\varepsilon\to 0$. Landau symbols $\mathcal{O}(\cdot)$
and $o(\cdot)$ will be used under the assumption $\varepsilon\to 0$
which will be often omitted after their use.

Suppose that $\rho(x)>0$ and $u(x)$ are in $C_b^1\big([R,\infty)\big)$.
Let $\rho_0,u_0\in\mathbb{R}$, $\rho_0>0$. Here, $C_b^1$ 
denotes a space of bounded functions with a bounded derivative. 
The initial data for (\ref{chRS}) are  
\begin{equation}\label{ri} (\rho,u)(x,0)=\begin{cases} (\rho_0,u_0), &
x\leq R\\ (\rho(x),u(x)), & x> R. \end{cases} 
\end{equation} 
Let us make a net of piecewise constant approximations
$(\rho^\varepsilon(x),u^\varepsilon(x))_{\varepsilon}$
of the initial data $(\rho(x),u(x))$.
Take a fixed $\varepsilon >0$ and a corresponding partition 
$\{Y_{i}\}_{i\in \mathbb{N}_{0}}$, 
$R:=Y_0<Y_1<Y_2<\ldots$, satisfying $Y_{i+1}-Y_i\leq \mu(\varepsilon)$, $i=0,1,\ldots$. The precise bound $\mu(\varepsilon)$ will be given in the proofs in
Section \ref{sec:analysis}.
The approximation is chosen such that
$\rho^\varepsilon(x)=\rho(Y_{i+1})=:\rho_{i+1}$,
$u^\varepsilon(x)=u(Y_{i+1})=:u_{i+1}$
for $x\in (Y_i,Y_{i+1}]$, $i\in \mathbb{N}_0$, and
$(\rho^\varepsilon(x),u^\varepsilon(x))=(\rho_0,u_0)$ for $x\leq R$.
Construction of a global solution 
is based on tracking wave fronts and analyzing interactions between waves. 
We need some preparations to do it. 

\begin{remark}
With a slight abuse of notation in the rest of the paper, we will use the
same notation ($u$ and $\rho$) for the initial function (which only
depends on space variable $x$) and for a solution (which depends on $x$
and $t$). A missing argument means that it equals $(x,t)$.
\end{remark}

\begin{definition}[Shadow waves] A shadow wave 
is a piecewise constant function with respect to time 
of the form 
\begin{equation}\label{sdw}
U^\varepsilon(x,t)=\begin{cases} (\rho_l,u_l), &
x<c(t)-a_\varepsilon(t)-x_{l,\varepsilon}\\
(\rho_{l,\varepsilon}(t),u_{l,\varepsilon}(t)), &
c(t)-a_\varepsilon(t)-x_{l,\varepsilon}<x<c(t)\\
(\rho_{r,\varepsilon}(t),u_{r,\varepsilon}(t)), &
c(t)<x<c(t)+b_\varepsilon(t)+x_{r,\varepsilon}\\ (\rho_r,u_r), &
c(t)+b_\varepsilon(t)+x_{r,\varepsilon}<x, \end{cases}
\end{equation} where
$a_\varepsilon(t)$, $b_{\varepsilon}(t)$, $x_{l,\varepsilon}$,
$x_{r,\varepsilon} \sim \varepsilon $. The states
$U_{\ast,\varepsilon}(t)=(\rho_{\ast,\varepsilon}(t),u_{\ast,\varepsilon}(t))$,
$\ast\in \{ l,r\}$ are called intermediate states. 
The curves $x=c(t)-a_\varepsilon(t)-x_{l,\varepsilon}$ and
$x=c(t)+b_\varepsilon(t)+x_{r,\varepsilon}$
are the external, while $x=c(t)$ is the central shadow wave line.
The limit $\lim_{\varepsilon\to 0}
\big((a_\varepsilon(t)+x_{l,\varepsilon})U_{l,\varepsilon}(t)
+(b_\varepsilon(t)+x_{r,\varepsilon})U_{r,\varepsilon}(t)\big)$
is the strength of shadow wave, while its speed is given by
$c'(t)$. Shadow waves with constant speed and constant 
intermediate values are called the simple ones.
Sometimes we use the prefix ``weighted'' for shadow waves with
variable intermediate state.
We say that (\ref{sdw}) solves (\ref{chRS}) in the approximated 
sense if its substitution into the right-hand side of the system
gives terms converging to zero as $\varepsilon \to 0$. 
\end{definition}

Let us note that in the case of system (\ref{chRS}) 
one can use that $U_{\varepsilon}(t)=U_{l,\varepsilon}(t)
=U_{r,\varepsilon}(t)$ without loss of generality, and we shall do it.
Also, note that all necessary calculations when (\ref{sdw}) is substituted
into (\ref{chRS}) can be done by using the classical Rankine-Hugoniot
conditions. In the sequel, we shall often skip the word ``approximate'' and
use only the word ``solution''. 

Approximation of the initial data using the partition 
$\{Y_{i}\}_{i\in \mathbb{N}_{0}}$ generates an infinite number of
Riemann problem for (\ref{chRS}) 
\begin{equation}\label{initial_class}
(\rho,u)(x,0)=\begin{cases} (\rho_i,u_i), & x<Y_i\\ (\rho_{i+1},u_{i+1}),
& x>Y_i \end{cases},\; i=0,1,2,\ldots.
\end{equation}

There are three kinds of solutions to (\ref{chRS}).
If $u_{i}=u_{i+1}$, a solution is a single contact discontinuity
\[ U(x,t):=(\rho,u)(x,t)=\begin{cases}
(\rho_i,u_i), & x-Y_i<u_i t\\ (\rho_{i+1},u_{i+1}), & x-Y_i>u_{i} t.
\end{cases}\]
It will be denoted by $\text{CD}_{i,i+1}$.
If $u_i<u_{i+1}$, solution to the Riemann problem is given by
\[ U(x,t)=\begin{cases} (\rho_i,u_i), &
x-Y_i<u_i t\\ (0,u_i(x,t)), & u_i t<x-Y_i<u_{i+1}t\\
(\rho_{i+1},u_{i+1}), & x-Y_i>u_{i+1} t, \end{cases} \]
with $u_i(x,t)$ being an arbitrary continuous function 
satisfying $u_i(Y_{i}+u_i t,t)=u_i $,
$u_i(Y_{i}+u_{i+1}t,t)=u_{i+1}$. Such solution is denoted by $\text{CD}_1^i+
\text{Vac}_{i,i+1}+\text{CD}_2^{i+1}$. 
Both of the above two solutions are classical and thus admissible. 
If $u_i>u_{i+1}$, the simple shadow wave 
\begin{equation}\label{sw}
U(x,t)=\begin{cases} (\rho_i,u_i), &
x-Y_i<\tilde{c}(t)-\frac{\varepsilon}{2}t\\
(\rho_{i,\varepsilon},u_{i,\varepsilon}), &
\tilde{c}(t)-\frac{\varepsilon}{2}t<x-Y_i<\tilde{c}(t)
+\frac{\varepsilon}{2}t \\ 
(\rho_{i+1},u_{i+1}), & x-Y_i>\tilde{c}(t)+\frac{\varepsilon}{2}t
\end{cases}
\end{equation} 
solves (\ref{chRS}). 
The shock is supported by the curve $c(t):=Y_i+\tilde{c}(t)>0$, 
where $\tilde{c}(0)=0$.
Strength of the wave is $\lim_{\varepsilon\to 0}\varepsilon
\rho_{i,\varepsilon}t$ and $\rho_{i,\varepsilon}\sim \varepsilon^{-1}$. 
More precisely, (\ref{sw}) satisfies system
(\ref{chRS}) in the approximated sense if the terms 
containing the $\delta|_{x=c(t)}$ are balanced:
\[ \begin{split} c'(t)(\rho_{i+1}-\rho_i) -
(\rho_{i+1} u_{i+1} -\rho_i u_i) &\approx \varepsilon
\rho_{i,\varepsilon}\\ c'(t)(\rho_{i+1} u_{i+1}-\rho_i u_i) -
(\rho_{i+1} u_{i+1}^2 -\rho_i u_i^2) &\approx\varepsilon
\rho_{i,\varepsilon} u_{i,\varepsilon}. 
\end{split} \]
The $\delta'$-terms are balanced if $c'(t)=u_s$.  Put
$u_s:=\lim_{\varepsilon\to 0} u_{i,\varepsilon}$ and $\xi
:=\lim_{\varepsilon\to 0}\varepsilon \rho_{i,\varepsilon}$. The above
imply that $\tilde{c}(t)=u_s t$, i.e.\ the speed of shadow wave is
constant, $c'(t)=u_s$. Also, 
\begin{equation}\label{ode_initial_sw}
\begin{split}
\xi&=u_s(\rho_{i+1}-\rho_i)-(\rho_{i+1}u_{i+1}-\rho_iu_i) \\
u_s\xi&=u_s(\rho_{i+1}u_{i+1}-\rho_iu_i)-(\rho_{i+1}u_{i+1}^2-\rho_iu_i^2).
\end{split}
\end{equation}
The system (\ref{ode_initial_sw}) reduces to
\[u_s^2(\rho_{i+1}-\rho_i)-2u_s(\rho_{i+1}u_{i+1}-\rho_iu_i)
+(\rho_{i+1}u_{i+1}^2-\rho_i u_i^2)=0.\]
If $\rho_{i+1}\neq \rho_i$, the solution of the above quadratic
equation is
\[u_s=\frac{\rho_{i+1}u_{i+1}-\rho_iu_i
\pm\sqrt{(\rho_{i+1}u_{i+1}-\rho_{i}u_i)^2-(\rho_{i+1}-\rho_i)
(\rho_{i+1}u_{i+1}^2-\rho_iu_i^2)}}{\rho_{i+1}-\rho_i}.\]

We say that wave (\ref{sw}) is overcompressive if $\lambda_{i}(\rho_{l},u_{l})
\geq u_{s} \geq \lambda_{i}(\rho_{r},u_{r})$, $i=1,2$. 
That will be true if we choose
the $+$ sign above ($u_s$ is a convex combination of $u_i$ and $u_{i+1}$).
So, if we denote $y_{i,i+1}:=u_s$, then overcompressibility condition
becomes
\begin{equation}\label{oc1} u_i\geq y_{i,i+1}\geq u_{i+1}\, 
\text{ and }\, y_{i,i+1}=\frac{\sqrt{\rho_{i+1}}u_{i+1}
+\sqrt{\rho_i}u_i}{\sqrt{\rho_{i+1}}+\sqrt{\rho_i}}.
\end{equation} 
Substituting $y_{i,i+1}$ in $(\ref{ode_initial_sw})$ one gets
that the strength of the shadow wave equals $\xi_{i,i+1}t$, where
$\xi_{i,i+1}:=\xi=\sqrt{\rho_i\rho_{i+1}}(u_i-u_{i+1})$. If
$\rho_{i+1}=\rho_i$, there exists unique solution to the system
(\ref{ode_initial_sw}) with $y_{i,i+1}=\frac{u_{i+1}+u_i}{2}$, 
$\xi_{i,i+1}=\rho_i(u_i-u_{i+1})$. The 
condition (\ref{oc1}) is satisfied in this case, too.

\section{The elementary interactions}\label{sec:interactions}

The first step in construction is the analysis of all
possible interactions between waves obtained after
the initial data approximation by step functions.

Suppose that two approaching waves interact.
Then the right state of the left incoming wave equals the left state of
the right incoming wave. That will be called the middle state in the 
interaction. So, the interaction problem including shadow waves 
can be viewed as an initial value problem containing the delta 
function.

\begin{lemma}\label{lemma:N1} 
Let (\ref{chRS}) with the initial data 
\begin{equation*}
(\rho,u)(x,0)=\begin{cases} (\rho_l,u_l), & x<X\\ (\rho_r,u_r), & x>X
\end{cases}+(\gamma,0)\, \delta_{(X,0)}, 
\end{equation*} 
be given, and denote $(\rho u)|_{t=0}=\tilde{\gamma}\delta_{(X,0)}$, 
where $u_l\geq\tilde{\gamma}/\gamma\geq u_r$, $\gamma>0$, 
$\rho_l,\rho_r\geq 0$. Then there exists an overcompressive 
shadow wave that solves the above initial data problem.
A strength $\xi(t)$ and a speed $u_s(t)$ are solutions to 
\begin{equation}\label{ode_lr} 
\begin{split} \xi'(t)
&=(\rho_r-\rho_l)u_s(t)-(\rho_r u_r- \rho_l u_l),\; \xi(0)=\gamma \\
(\xi(t) u_s(t))' &= (\rho_r u_r- \rho_l u_l)u_s(t)-(\rho_r u_r^2- \rho_l
u_l^2),\; \xi(0)u_s(0)=\tilde{\gamma}. 
\end{split} 
\end{equation} 
The front of the resulting shadow wave is given by 
$x=c(t):=\int_0^t u_s(\tau)\,d\tau + X$. 
\end{lemma} 

\begin{proof} Substitution of the shadow wave 
\[ U^\varepsilon(x,t)=\begin{cases}
(\rho_l,u_l), & x<c(t)-\frac{\varepsilon}{2}t-x_\varepsilon\\
(\rho_\varepsilon(t),u_\varepsilon(t)), &
c(t)-\frac{\varepsilon}{2}t-x_\varepsilon<x<c(t)
+\frac{\varepsilon}{2}t+x_\varepsilon\\
(\rho_r,u_r), & x>c(t)+\frac{\varepsilon}{2}t+x_\varepsilon
\end{cases} \]
into system (\ref{chRS}), where
$\rho_\varepsilon(t)\sim\varepsilon^{-1}$, $x_\varepsilon\sim \varepsilon$
and $u_s(t)=\lim_{\varepsilon\to 0}u_\varepsilon(t)$,
$\xi(t)=\lim_{\varepsilon\to 0}
2\big(\tfrac{\varepsilon}{2}t+x_\varepsilon\big)\rho_\varepsilon(t)$,
$c(0)=X$ reduces to system (\ref{ode_lr}) with the initial data 
$\xi(0)=\gamma$, $u_s(0)=\tilde{\gamma}/\gamma=:c$. 
The condition
$\xi(0)=\gamma$ is satisfied by choosing $x_\varepsilon$ such that
$\int_{X-x_\varepsilon}^{X+x_\varepsilon}\rho(x,0)\,dx=\gamma$.
That makes a distributional solutions being continuous in time.
Then, the solution is 
\begin{equation}\label{sol_xi} 
\begin{split} \xi(t) &=
\sqrt{\gamma^2+\rho_l\rho_{r}[u]^2t^2+2\gamma (c[\rho]-[\rho u])t} \\
u_s(t) &= \begin{cases} \frac{1}{[\rho]}\Big([\rho
u]+\frac{\rho_l\rho_{r}[u]^2t+\gamma(c[\rho]-[\rho u])}{\xi(t)}\Big),
& \text{ if } \rho_l\neq \rho_r\\
\frac{\gamma^2}{\xi^2(t)}(c-\frac{u_l+u_r}{2})+\frac{u_l+u_r}{2}, &
\text{ if } \rho_l= \rho_r,
\end{cases} 
\end{split} 
\end{equation}
where $[\cdot]:=\cdot_r-\cdot_l$ denotes a jump across a shock 
front. If $\rho_l\neq \rho_r$, we have
\begin{equation} \label{uu}
\begin{split}
u_s'(t) & 
= -\frac{[\rho]}{\xi(t)}(u_{s}(t)-y_{l,r})(u_{s}(t)-z_{l,r}) \text{ or} \\
u_s'(t) & =
-\frac{\gamma^2[\rho]}{\xi^3(t)}(c-y_{l,r})(c-z_{l,r}),
\end{split}
\end{equation}
where
\begin{equation}\label{y_{i,k}}
y_{l,r}:=\frac{u_l\sqrt{\rho_l}
+u_{r}\sqrt{\rho_{r}}}{\sqrt{\rho_l}+\sqrt{\rho_{r}}}, \;
z_{l,r}:=\frac{u_l\sqrt{\rho_l}
-u_{r}\sqrt{\rho_{r}}}{\sqrt{\rho_l}-\sqrt{\rho_{r}}}.
\end{equation}
Overcompressibility in the case $\rho_l\neq \rho_r$ follows from the fact
that $u_l\geq u_{s}(0) \geq u_r$. The functions $\rho$ and 
$\xi(t)$ are positive, and from the second line in (\ref{uu}) we have 
$\mathop{\rm sign}(u_s'(t))=-\mathop{\rm sign}([\rho]
(c-y_{l,r})(c-z_{l,r}))=-\mathop{\rm sign}(c-y_{l,r})$,
i.e.\ if $u_{s}(0)>y_{l,r}$, $u_{s}$ decreases. But it cannot go below
value $y_{l,r}$ because its derivative would be positive there due to
the first line in (\ref{uu}). The case $u_{s}(0)<y_{l,r}$
can be handled analogously. One can see that 
$\lim_{t\to \infty}u_s(t)=y_{l,r}$.
If $u_{s}(0)=y_{l,r}$, $u_{s}'$ is a constant, i.e.\ the shadow wave
has a constant speed. 
In any case, $u_s(t)\in [u_r,u_l]$ and the shadow wave is overcompressive. 
The proof in the case $\rho_l=\rho_r$ is similar. 
\end{proof}

\begin{remark}
The above lemma corresponds to Theorem 10.1 from \cite{mn2010},
so it can be used for (\ref{PGD}), too.
In that case the third component in the intermediate state
$U_\varepsilon(t)=(\rho_\varepsilon(t), u_\varepsilon(t), e_\varepsilon(t))$ 
satisfies $e_s(t)=\lim_{\varepsilon\to 0} e_\varepsilon(t)$ and
\begin{equation}\label{third}
c'(t)\Big[\rho \Big(\frac{u^2}{2}+e\Big)\Big]
-\Big[\rho u \Big(\frac{u^2}{2}+e\Big)\Big]
=\frac{d}{dt}\Big(\frac{u_s^2(t)}{2}\xi(t)+e_s(t)\xi(t)\Big).
\end{equation}
\end{remark}

\begin{remark} Note, one could not expect that (\ref{ode_lr}) 
can be explicitly solvable for some other systems admitting 
a shadow wave solution. 
\end{remark}

\begin{corollary}\label{coroll:bounds} With the above notation and
assumptions, we have
\begin{equation}\label{bounds_speed_strength}
\begin{split} & u_l\geq u_s(t)\geq u_r \text{ (overcompressibility
condition) and}\\ & \gamma+\min\{\rho_l,\rho_r\}(u_l-u_r)t\leq
\xi(t)\leq \gamma+\max\{\rho_l,\rho_r\}(u_l-u_r)t.
\end{split}
\end{equation}
\end{corollary}
\begin{proof} It follows from the proof of
Lemma \ref{lemma:N1}.
\end{proof}

Lemma \ref{lemma:N1} is used to solve the interaction problem. 
If the interaction occurs at the point $(X,T)$ the initial
data is translated to the interaction point, while the initial strength of
the resulting shadow wave is equal to the sum of strengths of
incoming waves at interaction time $t=T$. That is,
\begin{equation}\label{initial_strength}
\gamma=\xi(T)=\xi_l(T)+\xi_r(T),\end{equation}
where $\xi_l(t)$ and $\xi_r(t)$, $t<T$
are the strengths of the incoming waves. Also, denote by $u_{s_l}(t)$ and
$u_{s_r}(t)$, $t<T$ the speeds of incoming waves. Due to linear momentum
conservation the value $\tilde{\gamma}$ from Lemma \ref{lemma:N1} equals
$\tilde{\gamma}=\xi(T)u_s(T)=\xi_l(T)u_{s_l}(T)+\xi_r(T)u_{s_r}(T)$. Then
\begin{equation}\label{initial_speed}
c=u_s(T)=\frac{\xi_l(T)u_{s_l}(T)+\xi_r(T)u_{s_r}(T)}{\xi_l(T)+\xi_r(T)}.
\end{equation}

One can neglect the fact that interaction including at least one
shadow wave actually occurs a bit earlier. Let us show why.
Suppose that an interaction occurs between shadow waves with the external
shadow wave lines $x=c(t)\pm\frac{\varepsilon}{2}(t-\tilde{T})
\pm x_\varepsilon$ and
contact discontinuity $x=Y_i+u_{i+1} t$ at time $t=T$. 
The area bounded by the external shadow wave line
$x=c(t)+\frac{\varepsilon}{2}(t-\tilde{T})+x_\varepsilon$, the contact
discontinuity $x=Y_i+u_{i+1} t$, and the line $t=T$ is of the order
$\varepsilon^2$, and $\rho_\varepsilon(t)\sim \varepsilon^{-1}$. 
All terms of growth order less than $\varepsilon$ are neglected,
so one can neglect that area. Look at Figure \ref{f:area} for 
an illustration of the case when contact discontinuity is on 
the right-hand side. The situation is quite similar in the case
of a double shadow wave interaction.

The following lemma is based on the above arguments and will be used
repeatedly in the rest of the paper. For more details see Theorem 7.1 from
\cite{mn2010}.

\begin{lemma}\label{lemma:area} Let two approaching shadow waves with
the central lines given by $x=c_l(t)$ and $x=c_r(t)$ 
interact at time $t=\tilde{T}$. The value of $\tilde{T}$ is 
obtained by solving the equation
\[c_l(t)+\frac{\varepsilon}{2}(t-T_l)+x_{l,\varepsilon}=c_r(t)-\frac{\varepsilon}{2}(t-T_r)-x_{r,\varepsilon},\]
where $x=c_l(t)+\frac{\varepsilon}{2}(t-T_l)+x_{l,\varepsilon}$ is the 
right external SDW line of the first approaching shadow wave, while
$x=c_r(t)-\frac{\varepsilon}{2}(t-T_r)-x_{r,\varepsilon}$ is the 
left external SDW line of the second approaching shadow wave. 
Also, let $x_{l,\varepsilon},x_{r,\varepsilon}\sim \varepsilon$. 
A solution $T$ to $c_{l}(t)=c_{r}(t)$ will be called the interaction
time since the area bounded by two
external shadow wave lines and the line $t=T$ 
is of order $\varepsilon^2$ and all
terms of order $\varepsilon^\alpha$, $\alpha>1$ are neglected. Note that
$T=\tilde{T}+\mathcal{O}(\varepsilon)$.

The assertion stays true if one of the shadow waves is substituted by
a contact discontinuity.
\end{lemma}

\begin{figure}
\begin{center} 
\includegraphics*[scale=0.7]{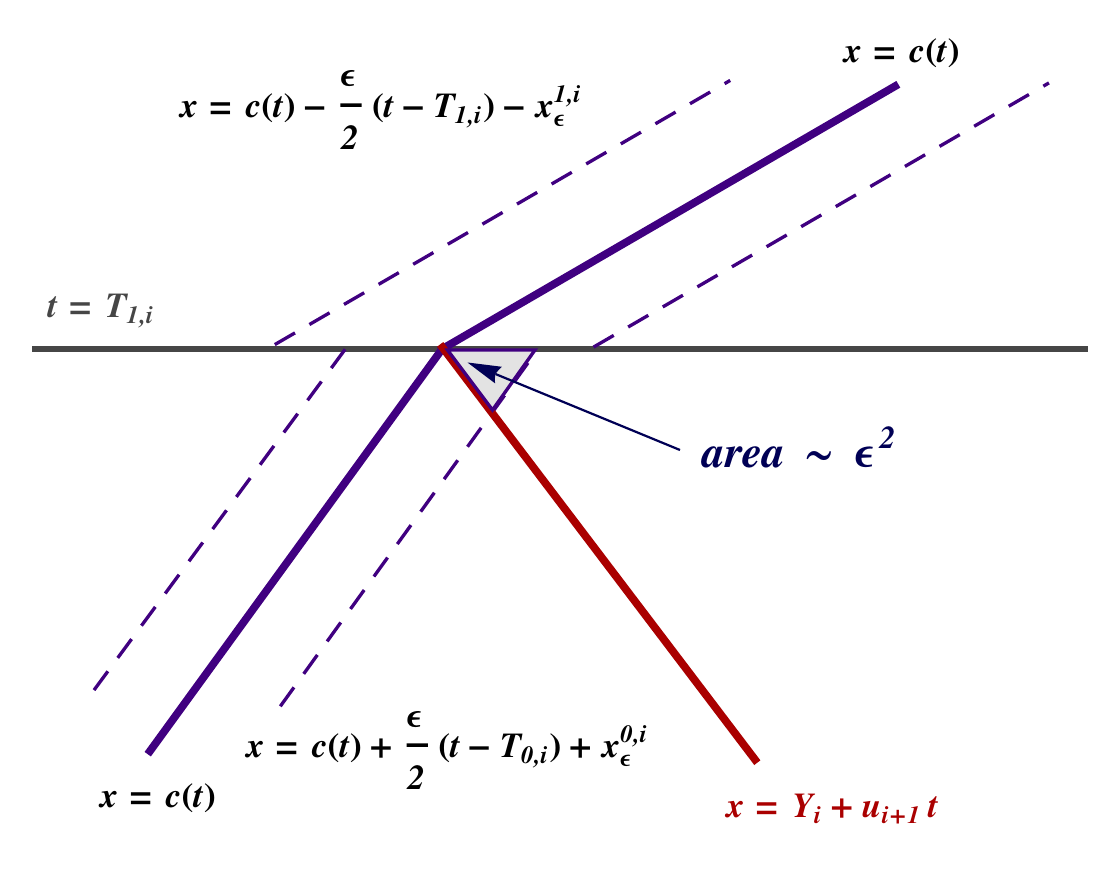}
\caption{Area bounded by the external SDW line, the contact
discontinuity and $t=T$}
\label{f:area}
\end{center}
\end{figure}

\begin{remark}\label{rmk:varepsilon-time} One should have in mind that a
phrase ``waves interact at the same time'' actually means that interactions
between those waves occur in the neglected area of order $\varepsilon^2$ 
described above.  That is, waves interact in a time interval of 
the order $\varepsilon$.
\end{remark}

\section{The algorithm}\label{sec:algorithm}
Let us fix some notation.
A shadow wave joining $(\rho_i,u_i)$ on the left and
$(\rho_j,u_j)$ on the right, $i<j$, 
$\rho_i,\rho_j> 0$ is denoted by $\text{SDW}_{i,j}$.

Let $i$ and $k$ be a given pair of indices. Then $\;^i \text{SDW}_{k}$,
$i\leq k$ denotes a shadow wave joining 
$\text{Vac}_{i-1,i}:=(0,u_{i-1}(x,t))$ on the left to 
$(\rho_k, u_k)$, $\rho_k> 0$ on the right. Note that
$\;^i\text{SDW}_{i}=\text{CD}_2^i$.

A shadow wave joining $(\rho_i,u_i)$, $\rho_i>0$ on its left to 
$\text{Vac}_{k,k+1}=(0,u_k(x,t))$ on its
right will be denoted by $\text{SDW}_i^k$, $i\leq k$. Again, 
$\text{SDW}_{i}^i=\text{CD}_1^i$. 

A shadow wave joining $(0, u_{i-1}(x,t))$ on the left
and $(0,u_{k}(x,t))$ on the right is denoted by
$\;^i \text{SDW}^k$.

\begin{remark} A wave $\text{SDW}_{i,j}$ exists only if $u_i\geq u_j$. 
Waves $\text{SDW}_l^r$ and $\;^l\text{SDW}_r$ are special solutions
to (\ref{sol_xi}). If $\rho_l> 0$ and $\rho_r=0$,  
\begin{equation}\label{sol_sdw_l^r}
\begin{split} \xi(t) & =\sqrt{\gamma^2 + 2\rho_l\gamma(u_l-c)t}\\ u_s(t)
& = u_l -\frac{\gamma(u_l-c)}{\sqrt{\gamma^2 +
2\rho_l\gamma(u_l-c)t}}\\ c(t) & =X+u_l t-\frac{1}{\rho_l}\xi(t) +
\frac{\gamma}{\rho_l}. 
\end{split} 
\end{equation} 
If $\rho_l=0$ and
$\rho_r>0$, the solution is given by (\ref{sol_sdw_l^r}), with $\rho_l$ and
$u_l$ replaced by $\rho_r$ and $u_r$. Finally, if $\rho_l=\rho_r=0$, the
resulting wave $\;^l\text{SDW}^r$ propagates with constant speed
and strength. 
\end{remark} 

A situation when three or more waves interact
at the same time in the sense of Remark \ref{rmk:varepsilon-time} 
is treated in the same way.
Suppose that there are $m$ incoming waves, $W_1,\ldots, W_m$. 
A resulting single wave depends on a state on the left to 
$W_{1}$, a state on the right of $W_{m}$, wave speeds 
and a sum of their strengths.
The middle states are lost in the interaction and 
there are the following possibilities.

\noindent {\rm (A1):} The wave $W_1$ has a left state 
$(\rho_l,u_l)$, $\rho_l>0$ and  
$W_m$ has a right state $\text{Vac}_{r,r+1}$.
The result is a single $\text{SDW}_l^r$, $l<r$.

\noindent
{\rm (A2):} The wave $W_1$ has a left 
state $(\rho_l,u_l)$, $\rho_l>0$ and $W_m$ has a right state
$(\rho_r,u_r)$, $\rho_r>0$. The result is a single $\text{SDW}_{l,r}$,
$l<r$.

\noindent
{\rm (A3):} The wave $W_1$ has a left state $\text{Vac}_{l-1,l}$
and $W_m$ a right state $(\rho_r,u_r)$, $\rho_r>0$.
The result is a single $\;^l\text{SDW}_{r}$, $l<r$.

\noindent
{\rm (A4):} The wave $W_{1}$ has a left state $\text{Vac}_{l-1,l}$ 
and $W_m$ a right state $\text{Vac}_{r,r+1}$. 
The result is a single $\;^l\text{SDW}^{r}$, $l<r$.

If the incoming waves are overcompressive, the resulting wave is
overcompressive, too. That follows from Corollary \ref{coroll:bounds} and
relation (\ref{initial_speed}). We are in a position to construct
an approximated solution.
\medskip

\noindent {\sc Algorithm:} \\
Suppose that given $\varepsilon$ is small enough. 

\noindent
{\sc Step 0}. Let $u_0\in\mathbb{R}$, $\rho_0>0$ be  
constants from (\ref{ri}). The set of initial states
$\{u_i\}_{i\in \mathbb{N}_{0}}$ 
and $\{\rho_i\}_{i\in \mathbb{N}_{0}}$ are sequences generated 
by the piecewise constant approximations of the functions 
$u(x)$ and $\rho(x)$, respectively, 
described in the paragraph below (\ref{ri}).

\noindent
{\sc Step 1}. Denote by $S_0:=\{U_k:\ k=0,1,2,\ldots\}$
the set of the initial states and by $I_0:=\{0,1,2,\ldots\}$
the set of corresponding indexes. A solution obtained by solving 
Riemann problems (\ref{chRS}, \ref{initial_class})
generated by states in $S_{0}$ is 
stopped at $t=T_1$ when the first interaction between two or
more waves occurs. If there are no interactions, 
all wave fronts continue to propagate to infinity
and the procedure finishes.  Each interaction between two or more waves
belongs to one of the four types (A1--A4) and gives a single 
shadow wave as a result. The resulting wave(s) 
as well as all other (non-interacting) waves
constitute a new set of states $S_1$ and a corresponding set of indexes 
$I_1\subset I_0$ after $t>T_{1}$.

\noindent
{\sc Step $j$ to $j+1$}. Suppose that $j$-th interaction occurs 
at a time $t=T_j$. Then we eliminate all 
middle states from $S_{j-1}$ and obtain a new
set $S_j$ and a corresponding $I_j=\{0,j_1,j_2,j_3,\ldots\}
\subset I_{j-1},\; 1\leq j_1<j_2<\ldots$. 
$k\in I_{j-1}\setminus I_j$ means that the state $U_{k}$ 
was a middle one in $S_{j-1}$.
All non-interacting waves are prolonged after $t>T_j$.
The procedure repeats with $j$ substituted by $j+1$
after a new interaction at $t=T_{j+1}$.
The algorithm stops when there is no $T_{j+1}$.
\medskip

It will be proved below that the procedure presented above gives a global
admissible solution to the problem (\ref{chRS}, \ref{ri}).

\begin{remark} \label{x}
	The above types (A1-A4) cover all possible interactions
between two or more waves. So the above procedure can also be applied
to the problem with initial data
\begin{equation*}
(\rho,u)(x,0)=\begin{cases} (\rho(x),u(x)), & x \leq R\\ (\rho_0,u_0), &
x>R, \end{cases}
\end{equation*} or any initial data
\begin{equation}\label{initial_3}
\rho(x,0)=\rho(x),\; u(x,0)=u(x), \;
x\in\mathbb{R},
\end{equation}
where $\rho(x)>0$, $u(x)$ having a finite number of jumps and
being piecewise $C_b^1\big(\mathbb{R}\big)$. 
\end{remark}

\section{Global existence and admissibility
of a solution}\label{sec:analysis}

The proof that our algorithm gives an admissible solution 
is divided into cases depending on monotonicity of a function $u(x)$
and relations between $u_{0}$ and $u(R)$.
A function $u(x)$ is called increasing (or decreasing) if 
$u(x)\leq u(y)$ (or $u(x)\geq u(y)$) for each $x<y$.
The function $u(x)$ is strictly increasing (or decreasing) 
if the inequality is strict.

\noindent
{\sc Case I}. $u(x)$ is increasing function for $x>R$ and $u_0\leq u(R)$.

This is a simple case with no interactions. The solution 
is a piecewise continuous function whose jumps are located along 
contact discontinuity lines.  
That is the consequence of the fact that $u_i\leq
u_{i+1}$ for each $i=0,1,\ldots$. Such waves never interact since the one
in front has a larger or the same speed.
\medskip

\noindent
{\sc Case II}. $u(x)$ is increasing function for $x>R$ 
and $u_0>u(R)$. 

Due to the boundedness assumption, there exists $\tilde{u}$, 
$\lim_{i\to\infty}u_i=\tilde{u}$. 
\begin{figure}[ht]
\begin{center}
\includegraphics*[scale=0.6]{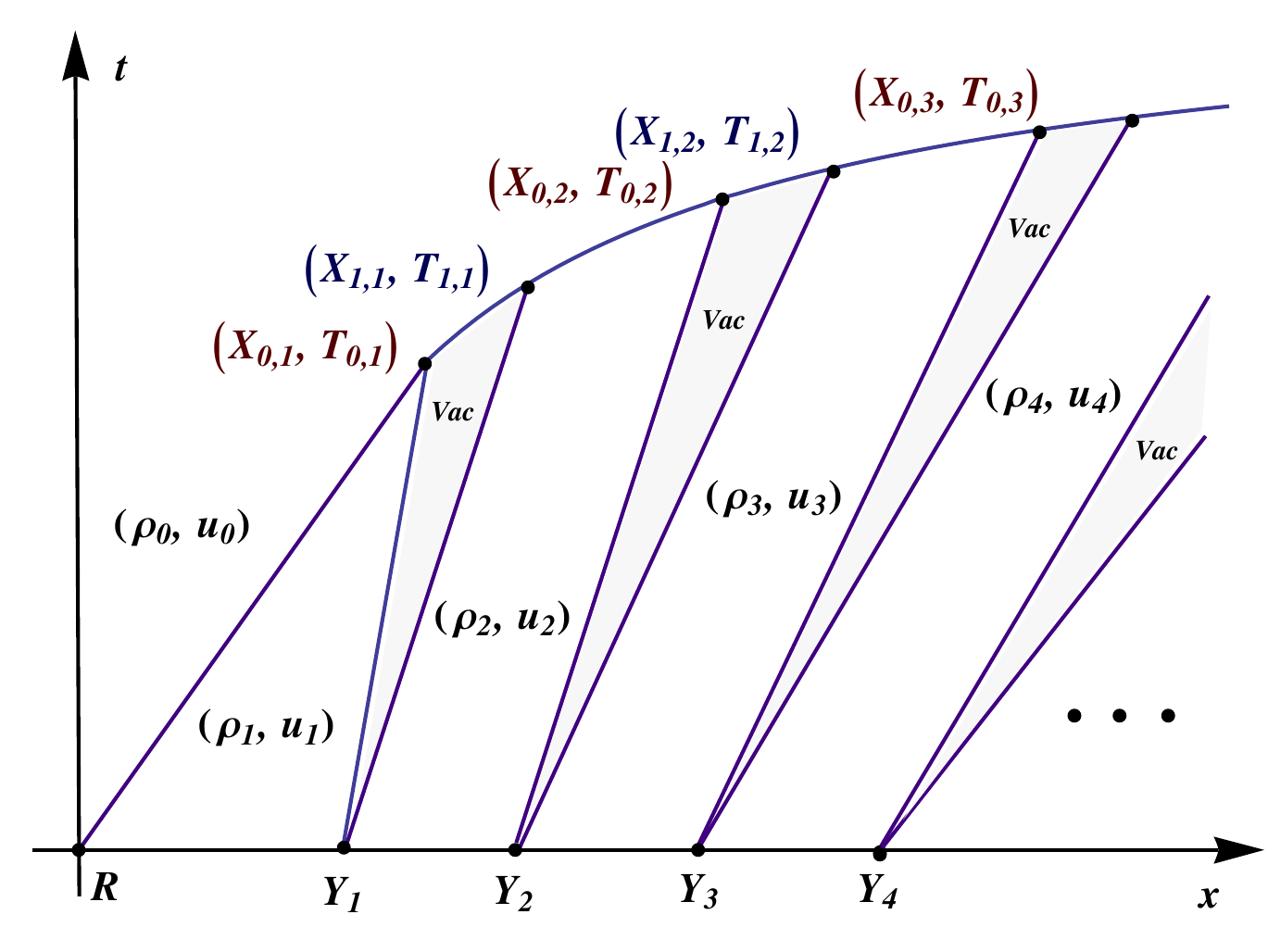}
\caption{Sketch of the interactions for strictly increasing $u(x)$ and
$\tilde{u}\leq u_0$} 
\end{center}
\end{figure}
The wave $\text{SDW}_{0,1}$ emanating from the point $(R,0)$ is a solution to 
(\ref{chRS}, \ref{initial_class}) for $i=0$. Solutions
to (\ref{chRS}, \ref{initial_class}) are 
$\text{CD}_1^i+\text{Vac}_{i,i+1}+\text{CD}_2^{i+1}$ emanating from
$(Y_i,0)$, $i=1,2,\ldots$. If $u_i=u_{i+1}$, the combination reduces to a single
$\text{CD}_{i,i+1}$. Note that all the interactions in this case
are of types (A1) or (A2). 
After each interaction exists only shadow wave
that started at $(R,0)$. Denote by $(X_{0,i},T_{0,i})$ a point where 
it meets the first contact discontinuity in the 
$i-$th wave combination $\text{CD}_1^i+\text{Vac}_{i,i+1}+\text{CD}_2^{i+1}$. 
$(X_{1,i},T_{1,i})$ is the interaction point of the shadow
wave and the second contact discontinuity.
The overcompressibility follows from Corollary
\ref{coroll:bounds} and interactions continue to infinity if
$\tilde{u}\leq u_0$ because of it.
If $\tilde{u}>u_0$, the solution is same until a point 
where the shadow wave enters the vacuum state and interactions stop,
again due to the overcompressibility.

The case of a single contact discontinuity when $u_i=u_{i+1}$ for some $i$
makes no real difference in the analysis. 
\medskip 

Let $U^{\varepsilon}=(\rho^{\varepsilon},u^{\varepsilon})$ be a function obtained by the above
procedure for a fixed $\varepsilon$. Denote by $\hat{U}^{\varepsilon}$ its singular 
part represented by the shadow wave approximation
\[ \hat{U}^{\varepsilon}(x,t)=\begin{cases} 
0, & x<c(t)-a_{\varepsilon}(t) \\
(\rho_{\varepsilon},u_{\varepsilon})(t), & 
c(t)-a_{\varepsilon}(t)<x<c(t)+a_{\varepsilon}(t) \\
0, & x>c(t)+a_{\varepsilon}(t)
\end{cases} \]
where 
\[a_{\varepsilon}(t)=\begin{cases}
\frac{\varepsilon}{2}t, & t\in (0,T_{0,1}] \\
\frac{\varepsilon}{2}(t-T_{0,i})+x_{\varepsilon}^{0,i},
& t\in (T_{0,i},T_{1,i}]\\
\frac{\varepsilon}{2}(t-T_{1,i})+x_{\varepsilon}^{1,i},
& t\in (T_{1,i},T_{0,i+1}]
\end{cases}, \; i=1,2,\ldots.\] 
Here, the values 
$(\rho_{\varepsilon},u_{\varepsilon})(t)$ and $x_{\varepsilon}^{k,i},$ $k=0,1$
are determined by Lemma \ref{lemma:N1} for each interval $(T_{0,i},T_{1,i}]$,  
$(T_{1,i},T_{0,i+1}]$, $i=1,2,\ldots$ separately.
That part of a solution is called 0-SDW and it approximates 
a weighted delta function with variable speed.

\begin{theorem}\label{theorem:global_sol} Let $u(x),\rho(x)\in
C_b\big([R,\infty)\big)$, $\rho(x)>0$. Assume that $u(x)$ is increasing
and let $\rho_0>0$ and $u_0>u(R)$. Take a partition
$\{Y_i\}_{i\in\mathbb{N}_0}$ of $[R,\infty)$, $Y_0=R$ such that
$C\sqrt[3]{\varepsilon}\geq Y_{i}-Y_{i-1}\geq \sqrt{\varepsilon}$ for every
$i=1,2,\ldots$ and a constant $C\geq 1$.
There exists an admissible global solution
to (\ref{chRS}, \ref{ri}), i.e.\ there exists a function
$U^\varepsilon=(\rho^\varepsilon,u^\varepsilon)$ satisfying
\begin{equation*}
\begin{split}
\partial_t \rho^{\varepsilon} +\partial_x(\rho^{\varepsilon}
u^{\varepsilon}) \approx
0,\; \partial_t(\rho^{\varepsilon} u^{\varepsilon})
+ \partial_x (\rho^{\varepsilon}
(u^{\varepsilon})^2) \approx 0,
\end{split}
\end{equation*}
$\rho^{\varepsilon}(x,0)\approx \rho(x,0)$,
$u^{\varepsilon}(x,0)\approx u(x,0)$ as
$\varepsilon \to 0$ and the admissibility condition.
\begin{enumerate}
\item If $\tilde{u}\leq u_0$, there are infinitely many interaction points.
\item If $\tilde{u}>u_0$, the interactions will stop with the
interaction point $(X_{0,k}, T_{0,k})$ where $k\in \mathbb{N}$ is
taken such that $u_k<u_0\leq u_{k+1}$ holds. In that case $u_s(t)\to
u_0$ as $t\to \infty$.
\end{enumerate}
\end{theorem}
\begin{remark}
One can use any $\mu(\varepsilon)\to 0,$ $\varepsilon\to 0$
instead of $C\sqrt[3]{\varepsilon}$ above. We have used that one
because of Theorem \ref{theorem:global_sol_mon} where $u(x)$
is not necessarily monotone. Also, any lower bound of order
$\varepsilon^\alpha,$ $0<\alpha<1$ can be used instead of
$\sqrt{\varepsilon}$ here.
\end{remark}
\begin{proof}
For a readers convenience we will present the complete proof here. 
Later on, we will skip technical details since they are 
similar to the ones in this proof.

\noindent
(1) Let $\tilde{u}\leq u_0$.
We have to prove that a solution $U^\varepsilon$,
$t\geq 0$, $x\in\mathbb{R}$ satisfies the following relations
\[\begin{split}
E_1 \!& :=\!\int_{0}^\infty\!\!\!\int_{-\infty}^\infty\!\!
\Big((\rho^{\varepsilon}\partial_t\varphi)(x,t)\!+\!(\rho^{\varepsilon}
u^{\varepsilon}\partial_x\varphi)(x,t)\Big)dxdt
\!+\!\int_{-\infty}^{\infty}(\rho^{\varepsilon}\varphi)(x,0)dx \approx 0\\
E_2 \!& :=\!\int_{0}^\infty\!\!\!\int_{-\infty}^\infty\!\!
\Big(\!(\rho^{\varepsilon} u^{\varepsilon}\partial_t\varphi)(x,t)\!+\!(\rho^{\varepsilon}
(u^{\varepsilon})^2\partial_x\varphi)(x,t)\!\Big)dxdt\!
+\!\int_{-\infty}^{\infty}\!\!(\rho^{\varepsilon}
u^{\varepsilon}\varphi)(x,0)dx \approx 0, 
\end{split}\] 
for every test function
$\varphi\in C_0^\infty\big(\mathbb{R}\times [0,\infty)\big)$.
We use the Taylor expansion of the test function $\varphi$, 
\begin{equation}\label{Taylor}
\begin{split}
&\varphi\big(c(t)-a_\varepsilon(t),t\big)
=\varphi\big(c(t),t\big)-\partial_x
\varphi\big(c(t),t\big)a_\varepsilon(t)
+\mathcal{O}(\varepsilon^2)\\
&\varphi\big(c(t)+a_\varepsilon(t),t\big)
=\varphi\big(c(t),t\big)+\partial_x
\varphi\big(c(t),t\big)a_\varepsilon(t)
+\mathcal{O}(\varepsilon^2)\\
&\varphi(x,t)=\varphi(c(t),t)+\mathcal{O}(\varepsilon)
\mbox{ for }x\in\big(c(t)-a_\varepsilon(t),
c(t)+a_\varepsilon(t)\big).
\end{split}\end{equation} 
Thus
\begin{equation*}\label{div_thm}
\begin{split} & \int_0^\infty \!\!\int_{-\infty}^\infty
(\rho^\varepsilon\partial_t \varphi)(x,t) \,dxdt =
I_0+\int_{-\infty}^\infty
\big((\rho^{\varepsilon}\varphi)(x,T_{0,1}-0)
-(\rho^{\varepsilon}\varphi)(x,0)\big)dx\\
&+\sum_{i=1}^{\infty}I_{0,i}+\sum_{i=1}^\infty \int_{-\infty}^\infty
\Big((\rho^{\varepsilon}\varphi)(x,T_{1,i}-0)
-(\rho^{\varepsilon}\varphi)(x,T_{0,i} + 0)\Big)\,dx\\ 
&+\sum_{i=1}^{\infty}I_{1,i}+\sum_{i=1}^\infty
\int_{-\infty}^\infty \Big((\rho^{\varepsilon}\varphi)(x,T_{0,i+1}-0)
-(\rho^{\varepsilon}\varphi)(x,T_{1,i}+0)\Big)\,dx,
\end{split} 
\end{equation*} 
where $I_0,I_{0,i}$ and $I_{1,i}$, are
integrals over $[0,T_{0,1}]$, $[T_{0,i},T_{1,i}]$ and $[T_{1,i},T_{0,i+1}]$,
respectively. All other terms cancel with the initial data and mutually
because we asked for a continuity of $U^\varepsilon$ with respect to $t$. 
In the same way the flux--part can be decomposed
$\int_0^\infty \int_{-\infty}^\infty \rho^\varepsilon u^\varepsilon
\partial_x \varphi\,dxdt=J_0+\sum_{i=1}^\infty J_{0,i}
+\sum_{i=1}^\infty J_{1,i}$,
where $J_0$, $J_{0,i}$ and $J_{1,i}$ are integrals over $[0,T_{0,1}]$,
$[T_{0,i},T_{1,i}]$ and $[T_{1,i},T_{0,i+1}]$, respectively. 
Note that we have finitely many
intervals due to the compactness of ${\rm supp}\, \varphi$.
If $u(x)$ is not strictly increasing, then some of the points $T_{0,i}$ and
$T_{1,i}$ would coincide. That does not influence the analysis.
\medskip

In the first interval $[0,T_{0,1}]$, we have 
\[\begin{split} 
I_0 =& \!-\!\int_{0}^{T_{0,1}}\!\!\rho_1 u_1\varphi\big(Y_1+u_1 t,t\big)dt\!-\!
\int_{0}^{T_{0,1}}\!\!(\rho_0-\rho_{0,\varepsilon})\varphi
\Big(\!R+y_{0,1}t-\frac{\varepsilon}{2}t,t\!\Big)\Big(y_{0,1}
-\frac{\varepsilon}{2}\Big)dt\\
&-\int_{0}^{T_{0,1}}(\rho_{0,\varepsilon} -\rho_1)
\varphi\Big(R+y_{0,1}t+\frac{\varepsilon}{2}t,t\Big)\Big(y_{0,1}
+\frac{\varepsilon}{2}\Big)dt\\ &+\underbrace{\sum_{i=1}^\infty
\int_{0}^{T_{0,1}} \rho_{i+1}u_{i+1}\big(\varphi\big(Y_i+u_{i+1}t,t\big)
-\varphi\big(Y_{i+1}+u_{i+1}t,t\big)\big)dt}_{:=A_0},
\end{split}\]
\[\begin{split}
J_0 :=& \int_0^{T_{0,1}}\!\!\int_{-\infty}^\infty \rho^\varepsilon
u^\varepsilon \partial_x \varphi\,dxdt =
\int_{0}^{T_{0,1}}\!\!\big(\rho_0u_0-\rho_{0,\varepsilon}
u_{0,\varepsilon}\big)\varphi\Big(R+y_{0,1}
t-\frac{\varepsilon}{2}t,t\Big)dt\\
&-\!\int_{0}^{T_{0,1}}\!\!\big(\rho_1u_1-\rho_{0,\varepsilon}
u_{0,\varepsilon}\big)\varphi\Big(R+y_{0,1}
t+\frac{\varepsilon}{2}t,t\Big)dt
+\int_{0}^{T_{0,1}}\!\!\rho_1u_1\varphi\big(Y_1+u_1 t,t\big)dt\\
&-\underbrace{\sum_{i=1}^\infty
\int_{0}^{T_{0,1}}\!\!\rho_{i+1}u_{i+1}\big(\varphi\big(Y_i+u_{i+1}t,t\big)-\varphi\big(Y_{i+1}+u_{i+1}t,t\big)\big)dt,}_{:=B_0}
\end{split}\]
since $\rho_{0,\varepsilon}$ does not depend on $t$. 
Using $\xi_{0,1}=y_{0,1}(\rho_1-\rho_0)-(\rho_1u_1-\rho_0 u_0)$,
(\ref{Taylor}) and the fact that $A_{0}=B_{0}$
we get
\[\begin{split} I_0+J_0
=&\int_0^{T_{0,1}}\!\!\big(\xi_{0,1}-\varepsilon
\rho_{0,\varepsilon}\big)\varphi\big(R+y_{0,1}t,t\big)\,dt\\
&- \int_0^{T_{0,1}}\!\!
t\varepsilon\rho_{0,\varepsilon}(y_{0,1}-u_{0,\varepsilon})\partial_x
\varphi\big(R+y_{0,1}t,t\big)\,dt+\mathcal{O}(\varepsilon).
\end{split}\]
From the fact that $\lim_{\varepsilon\to 0}u_{0,\varepsilon}=y_{0,1}$, 
$\lim_{\varepsilon\to 0}{\varepsilon \rho_{0,\varepsilon}}=\xi_{0,1}$, 
we have $E_1=\mathcal{O}(\varepsilon)$ in the strip $[0,T_{0,1}]$. 
The same relations with $\rho$ substituted by $\rho u$ and $\rho u$
by $\rho u^2$ give us $E_2=\mathcal{O}(\varepsilon)$ in $[0,T_{0,1}]$.
\medskip 

For $t\in [T_{0,i},T_{1,i}]$ we have the following relations
\[\begin{split} 
I_{0,i} =&-\int_{T_{0,i}}^{T_{1,i}}\int_{c(t)-\frac{\varepsilon}{2}
(t-T_{0,i})-x_\varepsilon^{0,i}}^{c(t)+\frac{\varepsilon}{2}(t-T_{0,i})
+x_\varepsilon^{0,i}}\partial_t \rho_\varepsilon(t)\varphi(x,t)\,dx\,dt\\
&-\int_{T_{0,i}}^{T_{1,i}}\big(\rho_0-\rho_\varepsilon(t)\big)
\varphi\Big(c(t)-\frac{\varepsilon}{2}(t-T_{0,i})-x_\varepsilon^{0,i},t\Big)
\Big(c'(t)-\frac{\varepsilon}{2}\Big)dt\\
&-\int_{T_{0,i}}^{T_{1,i}}\rho_\varepsilon(t)
\varphi\Big(c(t)+\frac{\varepsilon}{2}(t-T_{0,i})+x_\varepsilon^{0,i},t\Big)
\Big(c'(t)+\frac{\varepsilon}{2}\Big)dt\\
&+\underbrace{\sum_{k=i}^\infty \int_{T_{0,i}}^{T_{1,i}}
\rho_{k+1}u_{k+1}\Big(\varphi\big(Y_k+u_{k+1}t,t\big)
-\varphi\big(Y_{k+1}+u_{k+1}t,t\big)\Big)\,dt}_{:=A_{0,i}},
\end{split}\]

and
\[
\begin{split} J_{0,i} =&
\int_{T_{0,i}}^{T_{1,i}}\!\big(\rho_0u_0-\rho_\varepsilon(t)
u_\varepsilon(t)\big)\varphi\Big(c(t)-\frac{\varepsilon}{2}(t-T_{0,i})
-x_\varepsilon^{0,i},t\Big)dt\\
&+\int_{T_{0,i}}^{T_{1,i}}\!\rho_\varepsilon(t)
u_\varepsilon(t)\varphi\Big(c(t)+\frac{\varepsilon}{2}
(t-T_{0,i})+x_\varepsilon^{0,i},t\Big)dt\\
&-\underbrace{\sum_{k=i}^\infty \int_{T_{0,i}}^{T_{1,i}}\!\!
\rho_{k+1}u_{k+1}\Big(\varphi\big(Y_k+u_{k+1}t,t\big)
-\varphi\big(Y_{k+1}+u_{k+1}t,t\big)\Big)\,dt.}_{:=B_{0,i}=A_{0,i}}
\end{split}\] 
The proof that $I_{0,i}+J_{0,i}=\mathcal{O}(\varepsilon)$ for $t\in
[T_{0,i},T_{1,i}]$ follows from Lemma \ref{lemma:N1} and the method given
above. Again, we have $E_2=\mathcal{O}(\varepsilon)$ 
in the same interval by following the same arguments.
\medskip 

Finally, let $t\in [T_{1,i},T_{0,i+1}]$. Then 
\[\begin{split} 
I_{1,i} =&-\int_{T_{1,i}}^{T_{0,i+1}}\!
\int_{c(t)-\frac{\varepsilon}{2}(t-T_{1,i})-x_\varepsilon^{1,i}}^{c(t)
+\frac{\varepsilon}{2}(t-T_{1,i})+x_\varepsilon^{1,i}}\partial_t
\rho_\varepsilon(t)\varphi(x,t)\,dxdt\\
&-\int_{T_{1,i}}^{T_{0,i+1}}\!\big(\rho_0-\rho_\varepsilon(t)\big)
\varphi\Big(c(t)-\frac{\varepsilon}{2}(t-T_{1,i})-x_\varepsilon^{1,i},t\Big)
\Big(c'(t)-\frac{\varepsilon}{2}\Big)dt\\
&-\int_{T_{1,i}}^{T_{0,i+1}}\!\big(\rho_\varepsilon(t)-\rho_{i+1}\big)
\varphi\Big(c(t)+\frac{\varepsilon}{2}(t-T_{1,i})+x_\varepsilon^{1,i},t\Big)
\Big(c'(t)+\frac{\varepsilon}{2}\Big)dt\\
&-\int_{T_{1,i}}^{T_{0,i+1}}\!\rho_{i+1}u_{i+1}
\varphi\big(Y_{i+1}+u_{i+1}t,t\big)dt\\
&+\sum_{k=i+1}^\infty \int_{T_{1,i}}^{T_{0,i+1}}\!
\rho_{k+1}u_{k+1}\Big(\varphi\big(Y_k+u_{k+1}t,t\big)
-\varphi\big(Y_{k+1}+u_{k+1}t,t\big)\Big)dt,
\end{split}\] 
and 
\[\begin{split} 
J_{1,i} =& \int_{T_{1,i}}^{T_{0,i+1}}\!\!
\big(\rho_0u_0-\rho_\varepsilon(t)u_\varepsilon(t)\big)
\varphi\Big(c(t)-\frac{\varepsilon}{2}
(t-T_{1,i})-x_\varepsilon^{1,i},t\Big)dt\\
&+\int_{T_{1,i}}^{T_{0,i+1}}\!\!\big(\rho_\varepsilon(t)u_\varepsilon(t)
-\rho_{i+1}u_{i+1}\big)\varphi\Big(c(t)+\frac{\varepsilon}{2}(t-T_{1,i})
+x_\varepsilon^{1,i},t\Big)dt\\
&+\int_{T_{1,i}}^{T_{0,i+1}}\!\!\rho_{i+1}u_{i+1}
\varphi\big(Y_{i+1}+u_{i+1}t,t\big)dt\\
&-\!\sum_{k=i+1}^\infty \!\int_{T_{1,i}}^{T_{0,i+1}}\!\!\!
\rho_{k+1}u_{k+1}\Big(\!\varphi\big(Y_k+u_{k+1}t,t\big)\!
-\!\varphi\big(Y_{k+1}+u_{k+1}t,t\big)\!\Big)dt.
\end{split}\] 
The same arguments as above and Lemma \ref{lemma:N1} imply
$E_1=\mathcal{O}(\varepsilon)$. Proof for $E_2$ is the same.
\medskip 

Note that the proof holds even if $\rho_{i+1}=\rho_0$ for some $i$, since
$(\ref{sol_xi})_2$ implies $\xi'(t)=-\rho_0(u_{i+1}-u_0)$ and the expression
for $c'(t)$ does not have an influence in the proof.

Due to the fact that the test function $\varphi$ 
has a compact support and from $Y_i-Y_{i-1}\geq \sqrt{\varepsilon}$, 
one can see that there are at most 
$\frac{\mathop{\rm const}(\varphi)}{\sqrt{\varepsilon}}$ 
interactions. Thus, $E_1$ and $E_2$ are of order
$\mathcal{O}\big(\frac{1}{\sqrt{\varepsilon}}\big)
\mathcal{O}(\varepsilon)=\mathcal{O}(\sqrt{\varepsilon})$,
$\varepsilon \to 0$.
That proves the existence in the case $\tilde{u}\leq u_0$.
The admissibility of the obtained solution follows from the uniqueness of the
classical solutions and piecewise overcompressibility of the shadow wave 
in each segment. If $\tilde{u}=u_0$, then $y_{0,i+1}\to u_0$
as $i\to \infty$. That is, the overcompressibility implies that
the speed of the shadow wave is close to $u_0$ for $i$ large enough.
\medskip

\noindent (2) If $\tilde{u}>u_0$, then there exists $k\in \mathbb{N}$ such
that $u_{k+1}\geq u_0$ and $u_k<u_0$. Consequently, a curve $x=c(t)$ will
stay in vacuum area between two contact discontinuities emanating from $Y_k$
and the interactions will stop after the interaction point
$(X_{0,k},T_{0,k})$.
\end{proof}
\medskip

\noindent
{\sc Case III}. $u(x)$ is decreasing function for $x>R$ and $u_0\geq
u(R)$

The solution formed at the initial time is a piecewise constant
function with constant states connected by simple shadow waves. 
Each $\text{SDW}_{i,i+1}$
emanates from a point $Y_i$ and joints $(\rho_i,u_i)$ and
$(\rho_{i+1},u_{i+1})$. Thus, all possible cases of interactions are covered
by type (A2). With notation from (\ref{y_{i,k}}) we have
\begin{equation}\label{dif_decr}
\begin{split}
y_{i,k}-y_{k,j}=&\frac{\sqrt{\rho_i}}{\sqrt{\rho_i}\!+\!\sqrt{\rho_k}}\frac{\sqrt{\rho_k}}{\sqrt{\rho_k}\!+\!\sqrt{\rho_j}}(u_i\!-\!u_k)+\frac{\sqrt{\rho_i}}{\sqrt{\rho_i}\!+\!\sqrt{\rho_k}}\frac{\sqrt{\rho_j}}{\sqrt{\rho_k}\!+\!\sqrt{\rho_j}}(u_i\!-\!u_j)\\
&+\frac{\sqrt{\rho_k}}{\sqrt{\rho_i}+\sqrt{\rho_k}}\frac{\sqrt{\rho_j}}{\sqrt{\rho_k}\!+\!\sqrt{\rho_j}}(u_k\!-\!u_j)\geq
0 \text{ for } i<k<j,
\end{split}
\end{equation} since $u(x)$ decreases.
Due to overcompressibility each pair of 
shadow waves is approaching. 
The interaction point $(\tilde{X}_i,\tilde{T}_i)$
between $\text{SDW}_{i-1,i}$ and $\text{SDW}_{i,i+1}$, $i=1,2,\ldots$ is
determined by
\[\tilde{X}_i=Y_{i-1}+y_{i-1,i}\tilde{T}_i=Y_{i}+y_{i,i+1}\tilde{T}_i,\;
\tilde{T}_i=\frac{Y_i-Y_{i-1}}{y_{i-1,i}-y_{i,i+1}}.\]
(Note that $y_{i,i+1}$ is a speed of $\text{SDW}_{i,i+1}$.)
Then $\Delta u_i:=u_{i}-u_{i-1}<{\mathop {\rm const}_{u}} 
\mu(\varepsilon)$, $\sup_{x>R}|u'(x)| \leq \mathop {\rm const}_{u}$, 
since $\Delta Y_i:=Y_i-Y_{i-1}\sim \mu(\varepsilon)=:\mu$, $i=1,2,3,\ldots$.
Relation (\ref{dif_decr}) implies
\[y_{i-1,i}-y_{i,i+1}<(u_{i-1}-u_{i})+(u_{i-1}-u_{i+1})
+(u_i-u_{i+1})=2(u_{i-1}-u_{i+1})<{\mathop {\rm const}_{u}} 
\mu,\]
i.e.\ $\tilde{T}_i=\mathcal{O}(1)$. Note that it
is not possible to determine which interaction takes place the first and if
more then two shadow waves interact at the same time, since the relationship
between $y_{i-1,i}-y_{i,i+1}$ and $y_{i,i+1}-y_{i+1,i+2}$ depends on
$\rho(x)$, too.

If $u_0>u(R)$, the time of interaction between $\text{SDW}_{0,1}$ and
$\text{SDW}_{1,2}$ is of order $\mathcal{O}(\mu)$. So, the first
interaction that occurs is one between $\text{SDW}_{0,1}$ and
$\text{SDW}_{1,2}$. Resulting $\text{SDW}_{0,2}$ propagates until the next
interaction. Due to Lemma \ref{lemma:N1} the solution is
overcompressive for $t\geq 0$.

As in the Case II, 
0-SDW is defined to be the shadow wave connecting all piecewise 
defined $\text{SDW}_{0,i}$, $i=1,2,\ldots$ 
(see $\hat{U}^\varepsilon$ above).
Eventually, 0-SDW will overtake each $\text{SDW}_{i,i+1}$ and any other
shadow wave obtained by their mutual interactions.
All waves except the 0-SDW are called ``small'' shadow waves.
\medskip

\noindent
{\sc Case IV}. $u(x)$ is decreasing function for $x>R$ and $u_0< u(R)$

This case is similar to the previous one. The solution shortly
after the initial time consists of a wave combination
$\text{CD}_1^0+\text{Vac}_{0,1}+\text{CD}_2^1$ emanating from $x=R$
followed by a sequence
$\text{SDW}_{1,2}, \text{SDW}_{2,3},\ldots$. Possible types of interactions
are (A2) and (A3), and the order of interactions cannot be determined in
advance. Note that $\text{CD}_1^0$ does not interact with other waves,
while $\text{CD}_{2}^1$ interacts with some $\text{SDW}_{1,k}$. The
resulting $\;^1 \text{SDW}_{k}$ continues to propagate and collide
with shadow waves approaching from the right since it has a larger speed
than all waves from its right side. That conclusion follows from the
overcompressibility of the solution in each time interval. 

\begin{theorem}[Case III and IV]\label{theorem:global_sol_decr} 
Let $u(x),\rho(x)\in
C_b\big([R,\infty)\big)$, $\rho(x)>0$. Assume that $u(x)$ is decreasing
and let $\rho_0>0$ and $u_0\in\mathbb{R}$. Take a partition
$\{Y_i\}_{i\in\mathbb{N}_0}$ such that $C\sqrt[3]{\varepsilon}\geq Y_{i}-Y_{i-1}\geq
\sqrt[3]{\varepsilon}$, $i=1,2,\ldots$, $C\geq 1$ and $Y_0=R$. For  $\varepsilon >0$ small enough
there exists admissible global solution $U^\varepsilon$ to the
problem (\ref{chRS}, \ref{ri}).
\end{theorem}
\begin{proof} Denote by
$$S_0^{III}:=\{(\rho_i,u_i):\,i=0,1,2,\ldots\}$$ the set of initial
states for the case $u_0\geq u(R)$ and by
$$S_0^{IV}:=\{(\rho_0,u_0), \text{Vac}_{0,1}\}\cup
\{(\rho_i,u_i):\,i=1,2,\ldots\}$$ the set of initial states for the case
$u_0<u(R)$. Denote by $I_0^{\ast}=\{0,1,2,3,\ldots\}$ the
initial set of indexes corresponding to $S_0^{\ast}$,
$\ast\in\{III,IV\}$ as above. The analysis below is the same for both cases.

Suppose that an interaction occurs at $t=T_k$ for 
waves corresponding to states in $S_{k-1}^{\ast}$. 
A new set of states $S_k^\ast$ is constructed by eliminating all
the middle ones in interactions. The new set of indexes is now denoted by
$I_k^{\ast}=\{0,k_1^{\ast},k_2^{\ast},k_3^{\ast},\ldots\}$ 
where $1\leq k_1^\ast<k_2^\ast<k_3^\ast<\ldots$.

Let us prove that
\[\begin{split} \int_{0}^\infty\!\!\!\int_{-\infty}^\infty\!\!
\Big((\rho^{\varepsilon}\partial_t\varphi)(x,t)\!+\!(\rho^{\varepsilon}
u^{\varepsilon}\partial_x\varphi)(x,t)\Big)dxdt
\!+\!\int_{-\infty}^{\infty}(\rho^{\varepsilon}\varphi)(x,0)dx &\approx 0\\
\int_{0}^\infty\!\!\!\int_{-\infty}^\infty\!\! \Big(\!(\rho^{\varepsilon}
u^{\varepsilon}\partial_t\varphi)(x,t)\!+\!(\rho^{\varepsilon}
(u^{\varepsilon})^2\partial_x\varphi)(x,t)\!\Big)dxdt\!
+\!\int_{-\infty}^{\infty}\!\!(\rho^{\varepsilon}
u^{\varepsilon}\varphi)(x,0)dx &\approx 0,
\end{split}\]
for any $\varphi\in C_0^\infty\big(\mathbb{R}\times [0,\infty)\big)$

Again, values
$x_\varepsilon\sim\varepsilon$ are chosen such that the sum of strengths of
incoming waves is equal to the initial strength of outgoing shadow
wave (we use Lemma \ref{lemma:N1}). We proceed in the same way as in the
proof of Theorem \ref{theorem:global_sol}. \\ Put
\[\begin{split} Q_k &:=
\int_{T_k}^{T_{k+1}}\!\!\int_{-\infty}^\infty \rho^{\varepsilon} \,\partial_t
\varphi\,dxdt+ \int_{T_k}^{T_{k+1}}\!\!\int_{-\infty}^\infty
\rho^{\varepsilon}u^\varepsilon \,\partial_x \varphi\,dxdt.
\end{split}\]

Then,
\[\begin{split} \int_{0}^{\infty}\!\!\int_{-\infty}^\infty \rho^{\varepsilon}
\,\partial_t \varphi\,dxdt+ \int_{0}^{\infty}\!\!\int_{-\infty}^\infty
\rho^{\varepsilon}u^\varepsilon \,\partial_x \varphi\,dxdt=\sum_{k=0}^\infty Q_k.
\end{split}\]

It is enough to prove
$\varepsilon-$bounds for $Q_k$ due to Lemma \ref{lemma:N1}.
The sum is finite because
$\mathop{\rm supp} \varphi$ is compact.
Take two successive $i,j$ from $I_k^{\ast}$. There exists a
shadow wave or a contact discontinuity with the states corresponding to
indexes $i,j$. If it is a shadow one, denote 
its speed and strength by $u_s(t)$ and $\xi(t)$, respectively.
The intermediate state is denoted by
$(\rho_\varepsilon(t),u_\varepsilon(t))$. Then,
$Q_k$ is a sum of terms
\[\begin{split}
C_{i,j}^k:=\int_{T_{k}}^{T_{k+1}}\Big(A_{i,j}(t)-B_{i,j}(t)\Big)dt
+\int_{T_{k}}^{T_{k+1}}\!\!\int_{c(t)-\frac{\varepsilon}{2}(t-T)
-x_\varepsilon}^{c(t)+\frac{\varepsilon}{2}(t-T)+x_\varepsilon}\partial_t
\rho_\varepsilon(t)\varphi(x,t)\,dxdt,
\end{split}\]
where
\[\begin{split}
A_{i,j}(t):=&
(\rho_{i}-\rho_{\varepsilon}(t))\varphi\Big(c(t)-\frac{\varepsilon}{2}(t-T)
-x_\varepsilon,t\Big)\Big(c'(t)-\frac{\varepsilon}{2}\Big)\\
&+(\rho_{\varepsilon}(t)-\rho_{j})\varphi\Big(c(t)
+\frac{\varepsilon}{2}(t-T)+x_\varepsilon,t\Big)
\Big(c'(t)+\frac{\varepsilon}{2}\Big)\\
B_{i,j}(t):=& (\rho_{i}
u_{i}-\rho_{\varepsilon}(t)u_{\varepsilon}(t))\varphi\Big(c(t)
-\frac{\varepsilon}{2}(t-T)-x_\varepsilon,t\Big)\\
&+(\rho_{\varepsilon}(t)u_{\varepsilon}(t)-\rho_{j}u_{j})
\varphi\Big(c(t)+\frac{\varepsilon}{2}(t-T)+x_\varepsilon,t\Big),
\end{split}\]
for each pair $i,j$ of sequential indexes.
After some calculations analogous to the ones performed
in the proof of Theorem \ref{theorem:global_sol},
\[\begin{split}
C_{i,j}^k=&\int_{T_{k}}^{T_{k+1}}2\bigg(\partial_t \rho
_\varepsilon(t)\Big(\frac{\varepsilon}{2}(t-T)+x_\varepsilon\Big)
+\frac{\varepsilon}{2}\rho_\varepsilon(t)\bigg)\varphi\big(c(t),t\big)dt\\
&+\int_{T_{k}}^{T_{k+1}}\Big((\rho_{i}-\rho_{j})c'(t)
-(\rho_{i}u_{i}-\rho_{j}u_{j})\Big)\varphi\big(c(t),t\big)dt\\
&+\int_{T_{k}}^{T_{k}}2\rho
_\varepsilon(t)\big(c'(t)-u_\varepsilon(t)\big)
\Big(\frac{\varepsilon}{2}(t-T)-x_\varepsilon\Big)
\partial_x\varphi\big(c(t),t\big)dt+
\mathcal{O}(\varepsilon)=\mathcal{O}(\varepsilon).
\end{split}\]
If
$u_i,u_j$ are connected by a contact discontinuity, then $C_{i,j}^k=0$.
The condition $\Delta Y_i\geq \sqrt[3]{\varepsilon}$ ensures that at most
$\frac{\mathop{\rm const}(\varphi)}{\sqrt[3]{\varepsilon}}$ interactions occur,
since a test function $\varphi$ has a compact support. For the same
reason, solution in the interval $[T_k,T_{k+1}]$ consists of at
most $\frac{\mathop{\rm const}(\varphi)}{\sqrt[3]{\varepsilon}}$ wave fronts.
Then
\[\begin{split} Q_k =\frac{\mathop{\rm const}(\varphi)}
  {\sqrt[3]{\varepsilon}}
  \mathcal{O}(\varepsilon)=\mathcal{O}(\sqrt[3]{\varepsilon^2})
\text{ as } 
  \varepsilon \to 0.
\end{split}\]

The proof for the second equation
goes analogously ($\rho^{\varepsilon}$ is replaced by 
$\rho^{\varepsilon} u^{\varepsilon}$
and $\rho^{\varepsilon} u^{\varepsilon}$ by 
$\rho^{\varepsilon} (u^{\varepsilon})^{2}$). 
Admissibility of a solution follows from the
overcompressibility in each time interval $[T_{k},T_{k+1}]$. That concludes
the proof.
\end{proof}

\subsection{The general case}\label{sec:monot}

Suppose that a function $u(x)$ has a finite number of local extremes. We
will give a short analysis of the cases when a function has only one local
extremum. Cases when function changes monotonicity more than one
time can be treated in the same way.

\begin{figure}
\centering
	\begin{subfigure}[b]{0.49\textwidth}
	\begin{center}
    \includegraphics[width=5.6cm]{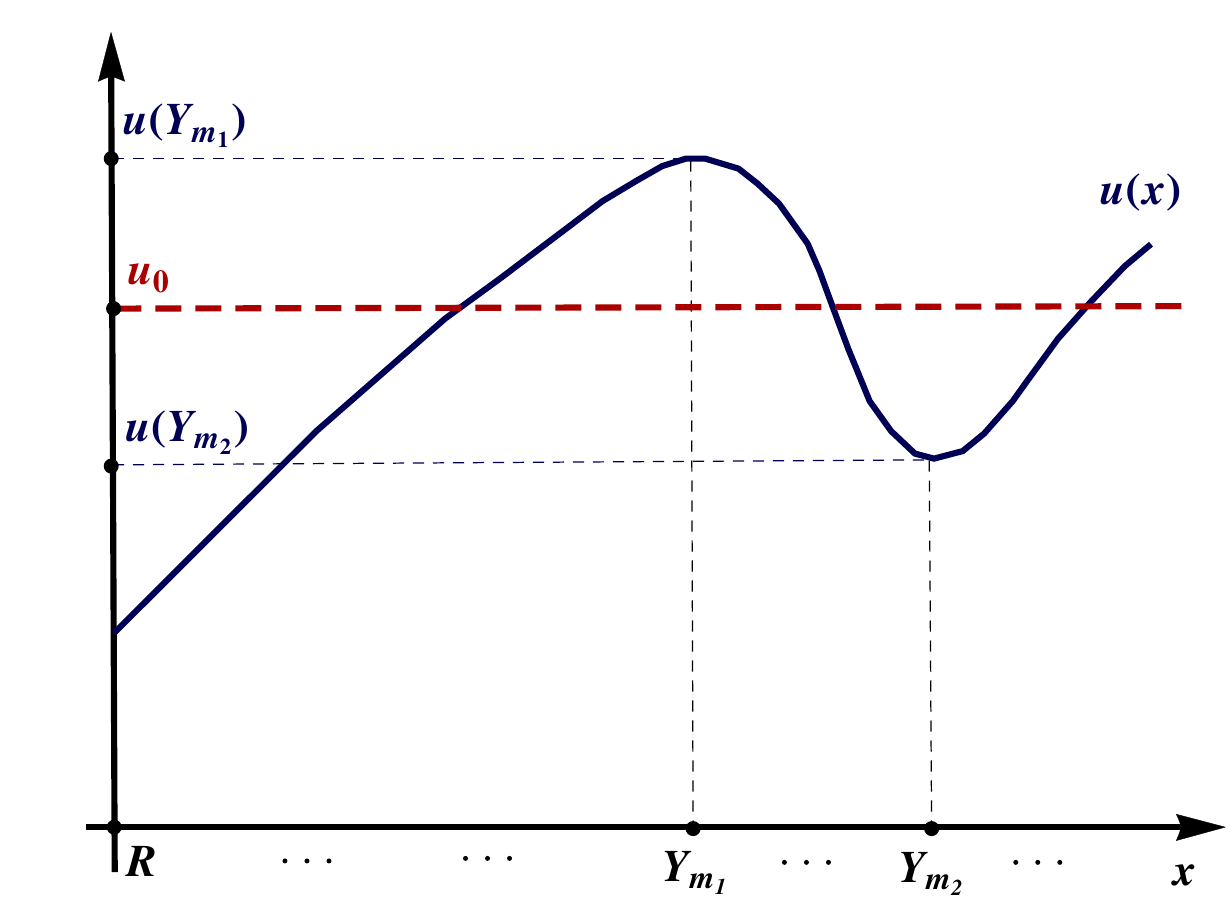}
    \caption{Sketch of $u(x)$}
	\label{f:a}
	\end{center}
	\end{subfigure}
    \begin{subfigure}[b]{0.49\textwidth}
    \begin{center}
    \includegraphics[width=5.6cm]{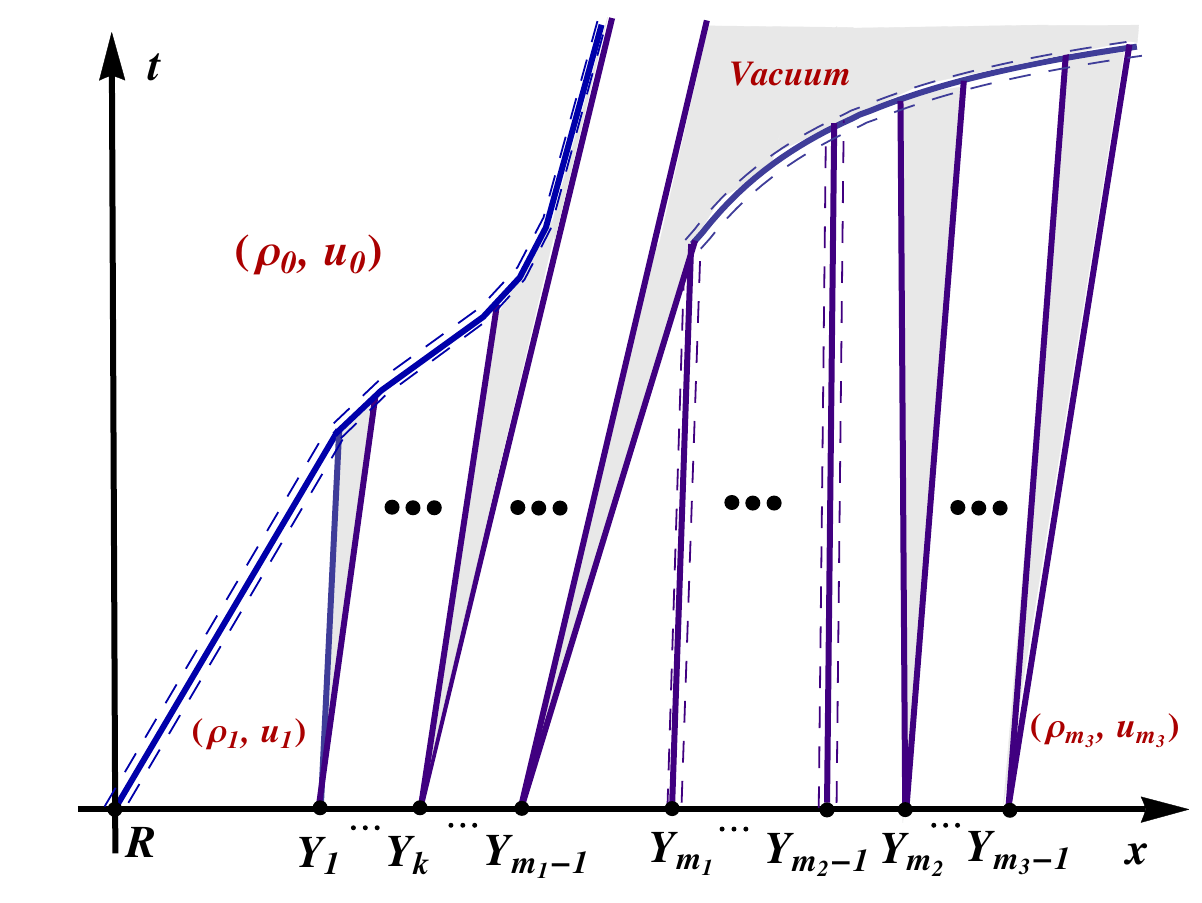}
    \caption{Sketch of interactions}
	\label{f:b}
	\end{center}
	\end{subfigure}
	\caption{Function $u(x)$  changes monotonicity}
	\label{f:incr_decr}
\end{figure}

Suppose that a local maximum of a piecewise constant approximation is
reached at a  point  $Y_{m_1}\in \{Y_i\}_{i\in
\mathbb{N}}$. 
If $u_0\leq u(R)$, the solution before the first
interaction consists of the combinations CD+Vac+CD
which do not interact with each other.
The last of them emanates from the point $(Y_{m_1-1},0)$. Starting from
the point $(Y_{m_1},0)$,  the solution is like
in Case III with $R=Y_{m_1}$ and $u_0=u(Y_{m_1})$.
Those waves continue to propagate until the first interaction. 
If $u_0>u(R)$, the solution
before it is a combination of waves obtained in 
Cases II and III. Unlike the case $u_0\leq u(R)$, the wave front
propagating from $(R,0)$ is $\text{SDW}_{0,1}$. 
Figure \ref{f:incr_decr} illustrates that case.

Similarly, if $u^{\varepsilon}(x)$ has a local minimum at 
$Y_{m_1}$ and $u_0\geq u(R)$,
the solution before the first interaction is a combination of 
shadow waves and contact discontinuities (Cases II and III): 
A sequence $\{\text{SDW}_{i,i+1}\}_{i=0}^{m_1-1}$ is 
followed by a sequence of wave combinations CD+Vac+CD.
If $u_0< u(R)$, a $\text{CD}_1^0+\text{Vac}_{0,1}+\text{CD}_2^1$
emanates from $(R,0)$ instead of  $\text{SDW}_{0,1}$ (Cases II and IV). 
\medskip

The proof of the following theorem will be omitted since
technical details are combined in proofs of 
Theorems \ref{theorem:global_sol} and \ref{theorem:global_sol_decr}.

\begin{theorem}[Global existence]
\label{theorem:global_sol_mon} Suppose that $u(x),\rho(x)\in
C_b\big([R,\infty)\big)$. Moreover, suppose that $u(x)$ has a finite
number of local extremes and that $\rho(x)>0$. Let $\rho_0>0$,
$u_0\in\mathbb{R}$ and consider a partition
$\{Y_i\}_{i\in\mathbb{N}_0}$, $Y_0=R$ such that $C\sqrt[3]{\varepsilon}\geq
Y_{i}-Y_{i-1}\geq \sqrt[3]{\varepsilon}$, $i=1,2,\ldots$, $C\geq 1$. For 
$\varepsilon>0$ small enough there exists an admissible global solution 
$U^\varepsilon$ to (\ref{chRS}, \ref{ri}) (in the approximated sense).
\end{theorem}

\begin{remark}
	One could easily check that the above statement also holds true 
	for (\ref{PGD}) when $e(x)\in C_b\big([R,\infty)\big)$ is positive.
	The energy variable does not have an influence on the approximated 
	solution behaviour.
\end{remark}

\section{Entropy dynamics and dissipation of energy}\label{sec:energy}

In the case of  system  (\ref{PGD}), the semi-convex entropy pair is given by
\begin{equation}\label{entropy3}
\eta(\rho,u,e)=\rho(R(u)+S(e)),\, Q(\rho,u,e)=\rho(R(u)+S(e)),
\end{equation}
where $R''(x)\geq 0$, $S'(x)\leq 0$ and $S''(x)\geq 0$ for each $x$ (see \cite{mn2010}).
The constructed solution should satisfy the entropy inequality 
$\partial_{t} \eta +\partial_{x} Q\leq 0$. 
Physically, it means that the mathematical entropy cannot increase. This condition is necessary and sufficient condition for uniqueness of $3\times 3$ pressureless gas dynamics system.

For the pressureless gas dynamics system (\ref{chRS}) it is known that using  semi-convex entropy pairs
$(\eta,Q)$ is not sufficient to extract a proper solution (we have to use overcompressibility). We will examine a dynamics of the physical energy and its flux,
\begin{equation}\label{energy_pair}
\eta(\rho,u)=\frac{1}{2}\rho u^2,\, Q(\rho,u)=\frac{1}{2}\rho u^3.\end{equation}
 

If shadow wave connecting states $U_l$ and $U_r$
emanating at the time $t=T_1$ satisfies the entropy condition, then
\begin{equation*} 
D_{l,r}(t):=-u_s(t)[\eta]+[Q]+\lim_{\varepsilon \to
0}\frac{d}{dt}\Big(\big(\varepsilon (t-T_1)+2x_\varepsilon\big)
\eta\big(U_\varepsilon(t)\big)\Big) \leq 0,
\end{equation*} 
as proved in \cite{mn2010}.
There is the second entropy condition given in 
the same paper, but it is always satisfied here
due to the fact that $u_{\varepsilon}(t)\approx u_{s}(t)$.
The overcompresibility implies $D_{l,r}(t)\leq 0$.

Here $D_{l,r}(t)$ is consistent with entropy production measure defined in  \cite{CD_1973} for general conservation law systems possessing bounded variation solutions.  We will call it the entropy production  across the
$\text{SDW}_{l,r}$ at time $t\geq T_1$. 
Denote by 
\[\mathcal{E}(t)=\int_{-M}^M \eta(U(x,t))dx\]
the total entropy at time $t$ of a solution $U(x,t)$. Here $M>0$ is taken to be large enough to avoid the total entropy being infinite in finite time.
\begin{theorem}\label{thm_entropy_decr}
Consider the system (\ref{chRS}) (or (\ref{PGD})). The total entropy  decreases after the interaction between two shadow waves.
\end{theorem}
\begin{proof}
Suppose that two shadow wave interact at time $t=T$. One from the left $\text{SDW}_{l,m}$ propagates with speed $u_{s_l}(t)$ and strength $\xi_l(t)$, while the right one $\text{SDW}_{m,r}$ propagates with speed $u_{s_r}(t)$ and  strength $\xi_r(t)$. (The corresponding specific internal energies are denoted by $e_{s_l}(t)$ and $e_{s_r}(t)$.) The speed and the strength of the resulting shadow wave $\text{SDW}_{l,r}$  are denoted by $u_s(t)$ and $\xi(t)$ (and the internal energy is $e_s(t)$).
The total entropy at time $t<T$ is given by $\mathcal{E}^-(t)$, while the total entropy at time $t>T$ (across $\text{SDW}_{l,r}$) is $\mathcal{E}^+(t)$. 

Consider first the system (\ref{chRS}) and entropy pair (\ref{energy_pair}). Then 
\[\begin{split}
\mathcal{E}^{+}(T+0)-\mathcal{E}^{-}(T-0) & =\frac{1}{2}\big(\xi(T)u_{s}^{2}(T)-\xi_{l}(T)u_{s_l}^{2}(T)-\xi_{r}(T)u_{s_r}^{2}(T)\big)\\
&= -\frac{1}{2}\frac{\xi_l(T)\xi_r(T)}{\xi(T)}(u_{s_l}(T)-u_{s_r}(T))^2\leq 0.
\end{split}\]
The above inequality follows from relations (\ref{initial_strength}) and (\ref{initial_speed}).

In the case of system (\ref{PGD}) and entropy pair (\ref{entropy3}) we have
\[\begin{split}
\mathcal{E}^{+}(T+0)-\mathcal{E}^{-}(T-0)  = &\xi(T)\Big(R(u_s(T))-\frac{\xi_{l}(T)}{\xi(T)}R(u_{s_l}(T))-\frac{\xi_{r}(T)}{\xi(T)}R(u_{s_r}(T))\\
&+S(u_s(T))-\frac{\xi_{l}(T)}{\xi(T)}S(u_{s_l}(T))-\frac{\xi_{r}(T)}{\xi(T)}S(u_{s_r}(T))\Big).
\end{split}\]
Using the relation
\[e_s(T)=\frac{\xi_l(T)}{\xi(T)}e_{s_l}(T)+\frac{\xi_r(T)}{\xi(T)}e_{s_r}(T)+\frac{1}{2}\frac{\xi_l(T)\xi_r(T)}{\xi^2(T)}(u_{s_l}(T)-u_{s_r}(T))^2\]
which follows from the continuity of energy across the interaction time,
and conditions imposed on functions $R$ and $S$, we get
\[\begin{split}
R(u_s(T)) &\leq \frac{\xi_{l}(T)}{\xi(T)}R(u_{s_l}(T))+\frac{\xi_{r}(T)}{\xi(T)}R(u_{s_r}(T))\\
S(e_s(T)) & \leq S\Big(\frac{\xi_l(T)}{\xi(T)}e_{s_l}(T)+\frac{\xi_r(T)}{\xi(T)}e_{s_r}(T)\Big)\\
 &\leq \frac{\xi_{l}(T)}{\xi(T)}S(u_{s_l}(T))+\frac{\xi_{r}(T)}{\xi(T)}S(u_{s_r}(T)).
\end{split}\]
Since $\xi(T)>0$ we have $\mathcal{E}^{+}(T+0)-\mathcal{E}^{-}(T-0) \leq 0$.
\end{proof}
\begin{remark}
The interaction between shadow wave and contact discontinuity can be treated as a special case of Theorem \ref{thm_entropy_decr}. It is enough to take the strength of the wave corresponding to contact discontinuity equal to zero. Then the entropy is constant across the interaction time.
\end{remark}

If the total entropy across 
 $\text{SDW}_{l,r}$ at time $t\geq T_1$ is denoted by $\mathcal{E}_{l,r}(t)$, the entropy rate is given by $\frac{d}{dt}\mathcal{E}_{l,r}(t)$ and the following relation holds
\begin{equation}\label{rel_sdw}
\frac{d}{dt} \mathcal{E}_{l,r}(t)=D_{l,r}(t)+Q(U_l)-Q(U_r).
\end{equation}

For (\ref{chRS}) and energy-entropy pair (\ref{energy_pair})  we can explicitly calculate the energy production,
\begin{equation}\label{D}
\begin{split}
D_{l,r}(t)
&=-\frac{1}{2}\big(\rho_l(u_l-u_s(t))^3+\rho_r(u_s(t)-u_r)^3\big)=:E(t).
\end{split}\end{equation}
The condition $D_{l,r}(t)\leq 0$ means that the energy is dissipative.
The value $\frac{d}{dt} \mathcal{E}_{l,r}(t)$ is called the energy dissipation rate   across 
 $\text{SDW}_{l,r}$ at time $t\geq T_1$.
Let 
\[\begin{split} A(t)&:=[\rho]u_s^2(t)-2[\rho u]u_s(t)+[\rho
u^2]=\rho_r(u_s(t)-u_r)^2-\rho_l(u_l-u_s(t))^2.
\end{split}\]
If
$\rho_l\neq \rho_r$, then $u_s'(t)=-A(T_1)\frac{\gamma^2}{\xi^3(t)}$ and
\[D_{l,r}'(t)=-\frac{3}{2}u_s'(t)A(t)
  =\frac{3}{2}\frac{\gamma^2}{\xi^3(t)}A(t)A(T_1).\]
Using $\xi(t)>0$ and the fact that $u_s(T_1) \geq y_{l,r}$ ($u_s(T_1) 
< y_{l,r}$) implies $u_s(t) \geq y_{l,r}$
($u_s(t) < y_{l,r}$, respectively) as proved in Lemma \ref{lemma:N1}, 
we get $D_{l,r}'(t)\geq 0$, $t>T_1$.
If $\rho_l=\rho_r\neq 0$, then
\[u_s'(t)=-2\frac{\gamma^2}{\xi^3(t)}
\Big(c-\frac{u_l+u_r}{2}\Big)\rho_l(u_l-u_r),\]
and $D'_{l,r}(t)\geq 0$. If $\rho_l=\rho_r=0$, then
$D_{l,r}(t)=D_{l,r}'(t)=0$ for $t>T_1$. So, 
$D_{l,r}(t)$ is non-positive and increasing function of time.  For a contact discontinuity $D_{l,r}$ is
equal to 0 (energy is conserved). 
Also, it is constant for a simple shadow wave since $u_s(t)$ 
does not depend on $t$.

If $Y_r-Y_l<C\sqrt[3]{\varepsilon}$ as needed in 
Theorem \ref{theorem:global_sol_mon}, 
then $D_{l,r}(t)=\mathcal{O}(\varepsilon)$. That means that dissipation
 across a small shadow wave is negligible.

Consider an interaction between $\text{SDW}_{l,r}$ and $\text{CD}_1^r$ at
$t=T_1$. The total energy production before the interaction equals
\[D_1(T_1-0)=-\frac{1}{2}\big(\rho_l(u_l-u_s(T_1-0))^3
+\rho_r(u_s(T_1-0)-u_r)^3\big),\]
and
\[D_1(T_1+0)=\frac{1}{2}\rho_l (u_s(T_1+0)-u_l)^3\]
after it. The speed continuity, $u_s(T_1-0)=u_s(T_1+0)$ implies
\[\Delta
D_1(T_1):=D_1(T_1+0)-D_1(T_1-0)=\frac{1}{2}\rho_r(u_s(T_1)-u_r)^3>0.\]
The resulting $\text{SDW}_l^r$ further interacts with $\text{CD}_2^{r+1}$ at
time $t=T_2$. Then
\[D_2(T_2-0)=\frac{1}{2}\rho_l (u_s(T_2-0)-u_l)^3\]
and
\[D_2(T_2+0)=-\frac{1}{2}\big(\rho_l(u_l-u_s(T_2+0))^3
+\rho_{r+1}(u_s(T_2+0)-u_{r+1})^3\big).\]
Thus,
\[\Delta
D_2(T_2)=-\frac{1}{2}\rho_{r+1}(u_s(T_2)-u_{r+1})^3<0.\] 

Let us consider an interaction between $\text{SDW}_{l,m}$ and
$\text{SDW}_{m,r}$ now. The $\text{SDW}_{l,m}$ propagates with a speed
$u_{s1}(t)$ and a strength $\xi_1(t)$, while $\text{SDW}_{m,r}$ propagates
with a speed $u_{s2}(t)$ and a strength $\xi_2(t)$. The initial speed of the
resulting $\text{SDW}_{l,r}$ equals
$u_s:=u_s(T)=\alpha u_{s1}+(1-\alpha)u_{s2}$, where
$\alpha:=\frac{\xi_1(T-0)}{\xi_1(T-0)+\xi_2(T-0)}$ and
$u_{si}:=u_{si}(T-0)$, $i=1,2$.
Then
\[\begin{split}
D(T-0)=&-\frac{1}{2}\big(\rho_m(u_{s1}-u_{m})^3+\rho_l(u_l-u_{s1})^3+\rho_r(u_{s2}-u_r)+\rho_m(u_m-u_{s2})^3\big)\\
D(T+0) =&-\frac{1}{2}\big(\rho_r(u_s-u_r)^3+\rho_l(u_l-u_s)^3\big)\\
\Delta D(T)=&-\frac{1}{2}(u_{s1}-u_{s2})\Big(\alpha
\rho_r\big((u_s-u_r)^2+(u_s-u_r)(u_{s2}-u_r)+(u_{s2}-u_r)^2\big)\\
&-\rho_m\big((u_{s1}-u_m)^2-(u_{s1}-u_m)(u_m-u_{s2})+(u_m-u_{s2})^2\big)\\
&+(1-\alpha)\rho_l\big((u_l-u_s)^2+(u_l-u_s)(u_{l}-u_{s1})+(u_{l}-u_{s1})^2\big)\Big).
\end{split}\]
Note that the sign of $\Delta D(t)$ depends on $\rho(x)$
and $u(x)$.
\begin{example} Suppose that $u(x)$ is a decreasing function, $u_0>u(R)$ and
$\rho(x)=\rho_0$ for each $x>R$. A simple $\text{SDW}_{i,i+1}$ emanating at the
$x-$axis  propagates with speed $y_{i,i+1}=\frac{u_{i}+u_{i+1}}{2}$
for every $i$. The result of an interaction at $t=T$ between 
$\text{SDW}_{i,i+1}$ and
$\text{SDW}_{i+1,i+2}$ is a new  $\text{SDW}_{i,i+2}$ with the constant speed
and strength given by
\[y_{i,i+2}=\frac{u_i+u_{i+2}}{2},\;
\xi_{i,i+2}t=\rho_0(u_i-u_{i+2})t,\,t\geq T.\]

It can be proved by an induction
that a solution in this case is piecewise constant function, with the
constant states connected by simple shadow waves, i.e.\ all jumps are
located along straight lines.
The energy production across $\text{SDW}_{l,m}$ and $\text{SDW}_{m,r}$ 
before and  after their interaction at $t=T$ is given by
\[D(T-0)=-\frac{\rho_l}{8}\big((u_l-u_m)^3+(u_m-u_r)^3\big),\;
D(T+0)=-\frac{\rho_0}{8}(u_l-u_r)^3.\]
Thus, in this case the energy dissipation rate decreases after the interaction (that follows from (\ref{rel_sdw})),
\[\Delta
D(T)=-\frac{3}{8}\rho_l(u_l-u_r)(u_l-u_m)(u_m-u_r)<0.\]
That is, the solution dissipates more energy after the interaction.
\end{example} 
When the pressure vanishes the entropy relation for gases $kTdS=dU+pdV,$ where $k$ is constant, $T$ is  a temperature, $S$ is an entropy, $U$ is an internal energy and $V$ is a volume, reduces to $S(e)={\rm const}\, e$ for a fixed temperature. Thus, let us put  $\eta(\rho,u,e)=-\rho e,$ $Q(\rho,u,e)=-\rho u e$ for the system (\ref{PGD}).
Then the entropy production is given by
\[\begin{split}
D_{l,r}(t) &=-u_s(t)(-\rho_r u_r+\rho_l u_l)+(-\rho_r u_r e_r+\rho_l u_l e_r) - \frac{d}{dt}(\xi(t)e_s(t))\\
&= \rho_r e_r(u_s(t)-u_r)+\rho_l e_l(u_l-u_s(t))-\frac{d}{dt}(\xi(t)e_s(t)).
\end{split}\]
Combining (\ref{ode_lr}) and (\ref{third}) it can be easily proved that 
\[\frac{d}{dt}(\xi(t)e_s(t))=-E(t)+\rho_r e_r(u_s(t)-u_r)+\rho_l e_l(u_l-u_s(t)),\]
 where $E(t)$ is defined in (\ref{D}). So, we have $D_{l,r}(t)=E(t),$ as for the $2\times 2$ system and the previous analysis also holds for (\ref{PGD}).

\section{Existence of a measure valued limit}\label{sec:limit}

A natural choice for a function space corresponding to our 
solution is the space of signed Radon measures due to the presence of delta
function. Radon measures are Borel regular and
locally finite measures, and can be understood as distributions of zero order.

We shall use the fact that for every signed measure $M$ there exist unique
nonnegative mutually singular measures $M^+$ and $M^-$ such that
$M=M^+-M^-$. Measures $M^+$ and $M^-$ are called positive and
negative variations of $M$ and $M=M^+-M^-$ is Jordan decomposition
of $M$ (see \cite{Conway} for details). The nonnegative measure
$|M|=M^{+}+M^{-}$ is called variation of $M$. The Riesz's
representation theorem gives the following characterization of the space of
signed Radon measures whose positive and negative variations are Radon
measures.
\begin{definition} \label{radonm} A space of signed Radon measures
$\mathcal{M}(\Omega)$ consists of linear forms $M$ defined on
$C_{0}(\Omega)$ such that for every compact set $K \subset \Omega$ there
exists a constant $C_{K}$ such that
\[ |\langle M,\varphi \rangle | \leq
C_{K} \|\varphi \|_{L^{\infty}}\; \text{ for all } \varphi \in
C_{0}(\Omega), \; \operatorname{supp}(\varphi) \subset K. \]

Denote by $\mathcal{M}_{f}(\Omega)$ the space of signed Radon
measures with a finite mass, i.e.\ $M\in\mathcal{M}_{f}(\Omega)$ if
there exist a constant $C$ such that
\[ |\langle M,\varphi \rangle |
\leq C \|\varphi \|_{L^{\infty}}\; \text{ for all } \varphi \in
C_{0}(\Omega). \]
\end{definition}

\begin{proposition}[Proposition 2.5.\ from
  \cite{dl}]\label{prop:dl} 
  Let $\{M_{\nu}\}_{\nu\in \mathbb{N}_0}$ be a sequence of
nonnegative uniformly locally bounded measures. Then there exists its
subsequence still denoted by $\{M_{\nu}\}_{\nu\in \mathbb{N}_0}$ 
and a Radon measure $M$ such
that $M_{\nu}\overset{\ast}{\rightharpoonup}M$.
\end{proposition}

\begin{theorem}[Existence of a weak limit]\label{thm:exist_limit} 
Suppose that $u(x),\rho(x)\in C_b\big([R,\infty)\big)$, $\rho(x)>0$,
$\rho_0>0$ and $u(x)$ 
having a finite number of local extremes.  Take any sequence
$\{\varepsilon_{\nu}\}_{\nu \in \mathbb{N}_0},$ $\varepsilon_\nu\to 0+$ 
satisfying 
$\sqrt[3]{\varepsilon_{\nu}}\leq Y_i-Y_{i-1}\leq 
C\sqrt[3]{\varepsilon_\nu}$, $C\geq 1$ for partition
$\{Y_{i}^\nu\}_{i\in \mathbb{N}_{0}}$
corresponding to $\varepsilon_{\nu}$. 
Denote $\{U^{\nu}\}_{\nu \in \mathbb{N}_{0}}$ a corresponding sequence
of solutions to problem (\ref{chRS}, \ref{ri}) constructed as in
Theorem \ref{theorem:global_sol_mon}.
There exists a subsequence still denoted by
$\{U^{\nu}\}_{\nu\in \mathbb{N}_0}$
and a signed Radon measure $U^{\ast}$ such that $U^{\nu}$
converges weakly to $U^{\ast}$ as $\nu \to \infty$.
\end{theorem}

To prove the existence of a limit $U^{\ast}$ we have to show that the
components of $|U^\nu|:=(\rho^\nu,|u^\nu|)$ are uniformly locally bounded
measures for each $\nu\in \mathbb{N}_0$. 
Note that $|\rho^\nu|=\rho^\nu$ since $\rho^\nu$
is nonnegative. The proof will rely on three lemmas given in the sequel.

\begin{remark} Note that we will not emphasize that
$U^\nu$, as well as $U^\ast$ are vector-valued measures since one can
easily distinguish vector from scalar valued measures.
\end{remark}

\begin{lemma}[Finite propagation speed]\label{lemma:speed_bound}
  Suppose that $\rho(x)$ and $u(x)$ are continuous and bounded 
  functions and $u(x)$ has a finite number of local extremes. 
  Then a speed of any wave which is part of an admissible solution to 
  problem (\ref{chRS}, \ref{ri}) is bounded.
\end{lemma} 

The proof of the above Lemma is straightforward.
Each shadow wave is overcompressive,
while a speed of each contact discontinuity is constant that equals to
a value of $u(x)$ at some point $x>R$. Thus,  
the propagation speed is between
$\min\big\{u_0,\inf_{x\geq R}u(x)\big\}$ and 
$\max\big\{u_0,\sup_{x\geq R} u(x)\big\}$.

\begin{lemma}\label{lemma:strength_bound} 
  Let $U^{\nu}$ be the admissible solution to (\ref{chRS}, \ref{ri}), 
  with $u(x)$ and $\rho(x)$ satisfying the assumptions from 
  the previous lemma. Then 
  \[\begin{split}
      \underbrace{\inf_{i \in \mathbb{N}_{0}}\xi_i t|_{t=0}}_{=0}
      < \xi(t) < 
      \underbrace{\sup_{i \in \mathbb{N}_{0}}\xi_i t|_{t=0}}_{=0} +  
      \overline{\rho}
      \Big(\max\big\{u_0,\sup_{x\geq R}u(x)\big\}
      -\min\big\{u_0,\inf_{x\geq R}u(x)\big\}\Big) t,
  \end{split}\] 
  for some $\overline{\rho}$, 
  where $\xi_i t$ is the strength of $i$-th wave emerging at the initial time,
  $i\in \mathbb{N}_{0}$.
\end{lemma}
\begin{proof} 
Define $\bar{\rho}:=\max\big\{\rho_0,\sup_{x\geq R} 
\rho(x)\big\}$. Let $c_{l}$ (or $c_{r}$) and 
$\sigma_{l}$ (or $\sigma_{l}$) be a speed and a strength 
of an incoming wave from left (or right). Let $t=T$ be a time of 
the interaction. Then, the initial speed and
the strength of the resulting wave are
\[c:=u_s(T+0)=\frac{\sigma_lc_l+\sigma_r c_r}{\sigma_l+\sigma_r}, \;
\sigma:=\xi(T+0)=\sigma_l+\sigma_r.\]
The global bounds for a strength of any wave propagating at time $t$
follow from estimate (\ref{bounds_speed_strength})$_{2}$.
\end{proof}

\begin{lemma}\label{lemma:lbound_meas} Suppose that all assumptions of
Theorem \ref{thm:exist_limit} hold. Denote by
$\{U^\nu\}_{\nu\in\mathbb{N}_0}$ the sequence
defined in that theorem.  Then $\rho^\nu$ and $|u^\nu|$ are (nonnegative)
uniformly locally bounded measures for each $\nu\in\mathbb{N}_{0}$.
\end{lemma}
\begin{proof} Due to construction of the solution, boundedness
of $u(x)$ and Lemma \ref{lemma:speed_bound} we have that $u^\nu$ is
uniformly globally bounded function for $\nu \in \mathbb{N}_0$. 
In order to prove that $\rho^\nu$ is uniformly $L^1_{loc}$-bounded 
for $\nu \in \mathbb{N}_0$,
we will use the conservation of mass principle,
boundedness of $\rho(x)$ and the finite propagation speed property.
For each $E\Subset \mathbb{R}$ there exists a $C_E>0$ such that
\[0\leq \int_{E\times (t_0,T)}\rho^\nu(x,t)\,dxdt\leq (T-t_0) \cdot C_E
\sup_{x\in \mathbb{R}}\rho(x,0)<\infty.\]
Thus, $|u^\nu|$ and $\rho^\nu$
are bounded in $L^{1}(K)$ for every compact set $K\subset \mathbb{R}^2_+$,
i.e.\ $|U^\nu|$ is uniformly locally bounded measure.
\end{proof}

\begin{proof}[Proof of Theorem \ref{thm:exist_limit}] Due to Lemma
\ref{lemma:lbound_meas} we know that $\rho^\nu$ and $|u^\nu|$ are
(nonnegative) uniformly locally bounded measures. Thus, there exist
uniformly locally bounded measures $U^\nu_+$ and $U^\nu_-$ such that
$U^\nu=U^\nu_+-U^\nu_-$ and $|U^\nu|=U^\nu_{+} + U^\nu_-$. From Proposition
\ref{prop:dl} it follows that there exist subsequences 
$\{U^\nu_{+}\}_{\nu \in \mathbb{N}_0}$,
$\{U^\nu_{-}\}_{\nu \in \mathbb{N}_0}$ 
and locally finite measures $U^{\ast}_+$, 
$U^{\ast}_{-}$ such that $U^\nu_+\overset{\ast}{\rightharpoonup}
U^{\ast}_+$ and $U^\nu_-\overset{\ast}{\rightharpoonup} U^{\ast}_-$. Thus,
$U^\nu$ converges weakly to $U^{\ast}:=U^{\ast}_+-U^{\ast}_-$. Note that
one can also use Proposition \ref{prop:dl} directly to obtain the
subsequence $\{|U^\nu|\}_{\nu \in \mathbb{N}_0}$ 
that converges weakly to $|U^{\ast}|$.
\end{proof}

In certain cases, it is possible to find an explicit form of
a measure--valued limit $U^*$ at least for some small time interval
as one can see in the following theorem.

\begin{theorem} \label{existence} 
Suppose that  all the assumptions of
Theorem \ref{thm:exist_limit} hold, as well as the notation. Let $u_{0}>u(R)$. 
There exists $T_{\mathop{\rm max}}>0$ such that
$U^{\ast}$ is the weighted $\delta$ measure supported by a curve 
$\Gamma:\; x=c(t)$ that connects $U_{0}$ from the left and
a classical solution $U(x,t)$ to (\ref{chRS}) to the right 
in the strip $t<T_{\max}$.
The life-span $T_{\mathop{\rm max}}$ 
is a positive infimum of $-\frac{1}{u'(x)}$, $x>R$
such that $D_{x}:=\Big( x-\frac{u(x)}{u'(x)},-\frac{1}{u'(x)}\Big)$
lies above the curve $\Gamma$.
\end{theorem}
\begin{remark} Theorem
\ref{existence} holds for $u_0\leq u(R)$ and increasing $u(x)$ too.
That is a trivial case since a solution converges to a
smooth solution obtained by the method of characteristics.
\end{remark}
\begin{proof} Let $T>0$ be arbitrary but fixed.
First, we will show that $\hat{U}^{\nu}$ 
has a subsequence that converges. 
 It is bounded in 
$L^{1}_{\rm loc}(\mathbb{R}_{+}^{2})$ uniformly for $\nu \in \mathbb{N}_0$
by Lemmas \ref{lemma:speed_bound} and
\ref{lemma:strength_bound}. 
Therefore, it has a subsequence
that converges to some $\hat{U}^{\ast}\in \mathcal{M}(\mathbb{R}_+^2)$.
From the construction, it is obvious that its support is the curve
$\Gamma$.

On the other hand, a part of $U^\nu$ lying to the right of
$\hat{U}^{\nu}$ converges to a classical solution $U$ obtained by 
method of characteristics as long the classical solution exists.
Let us show that. 

Suppose that $u(x)$ is increasing.
The procedure from Section \ref{sec:algorithm} gives the 
admissible solution $U^\nu=(\rho^\nu,u^\nu)$ to (\ref{chRS}) 
consisting of a sequence of contact discontinuities connected by a 
vacuum state. The classical initial value problem
$(\rho,u)|_{t=0}=(\rho(x),u(x))$
can be solved by method of characteristics. For smooth solutions and
away from vacuum state one gets the Burgers equation $\partial_t
u+u\partial_x u=0$. Its characteristics are integral curves of ordinary
differential equation $\frac{dx}{dt}=u(x(t),t)$ and a solution is given by
\begin{equation*} 
	u(x,t)=u(\psi(x,t)),
\end{equation*} where a
function $\psi=\psi(x,t)$ satisfies $x=u(\psi)t+\psi$. The existence of
function $\psi$ for each $t>0$ and in the region where $u(x)$ is strictly
increasing follows from the Implicit Function Theorem.
From the first equation in (\ref{chRS}), one can see that $\rho$
satisfies the equation $\partial_t \rho+u\partial_x \rho=-\rho \partial_x
u$. That is,
\begin{equation*} 
\rho(x,t)=\rho(\psi(x,t))\exp\Big(\!-\!\int_0^t
\frac{u'(\psi(x,t))}{u'(\psi(x,t))s+1}\,ds\Big)\in C^1.
\end{equation*}
The solution $(\rho,u)$ corresponding to the region where $u(x)$ is constant
is also constant. 

For each interval $[X_{-},X_{+}]$ and time $T>0$, let us show that
\begin{equation}\label{limit_incr}
\begin{split}
I_{\nu}:=\int_{X_{-}}^{X_{+}}\rho^{\nu}(x,T)\,dx & \to
\int_{X_{-}}^{X_{+}}\rho(x,T)\,dx \text{ as } \nu \to \infty, \\
\int_{X_{-}}^{X_{+}}u^{\nu}(x,T)\,dx & \to
\int_{X_{-}}^{X_{+}}u(x,T)\,dx \text{ as } \nu \to \infty.
\end{split}
\end{equation} 
For any $\nu\in \mathbb{N}_0$, let $U^\nu$ be the solution
constructed by using the partition $\{Y_i\}_{i\in \mathbb{Z}}$ 
such that $Y_i-Y_{i-1}<C\sqrt[3]{\varepsilon_{\nu}}$, $C\geq 1$,
$i\in \mathbb{Z}$. (To simplify notation we will drop superscript
$\nu$ in $\{Y_{i}^{\nu}\}_{i\in\mathbb{{Z}}}$.)
There exist $Y_{-}$, $Y_{+}$ such that
$X_{-}={Y_-}+u(Y_{-})T$ and ${X_+}=Y_{+}+u(Y_{+})T$. Suppose that
$Y_{-}\in(Y_{l-1},Y_{l}]$, ${Y_+}\in[Y_{m},Y_{m+1})$ for some
$l,m \in \mathbb{Z}$. Denote by
$X_{0,i}:=Y_{i}+u(Y_{i})T$, $X_{1,i}:=Y_{i}+u(Y_{i+1})T$,
$i\in \mathbb{Z}$.
The function $u^{\nu}(x,t)$ is a good approximation of $u(x,t)$
since it is uniquely determined in non-vacuum part and 
its value in vacuum part is continuously interpolated.
 We will use the conservation of
mass to prove $(\ref{limit_incr})_1$. Note that $\rho^{\nu}(x,T)=0$,
$x\in(X_{0,i},X_{1,i})$ and 
\[\begin{split} I_{\nu} = &
\int_{X_{-}}^{X_{+}}\rho^{\nu}(x,T)\,dx
=\int_{X_{0,l}}^{X_{1,m}}\rho^{\nu}(x,T)\,dx +\underbrace{
\int_{X_{-}}^{X_{0,l}}\rho^{\nu}(x,T)\,dx +
\int_{X_{1,m}}^{X_{+}}\rho^{\nu}(x,T)\,dx}_{=:I_{\nu}'}\\ = &
\sum_{i=l}^{m-1}\rho_{i+1}\big(X_{0,i+1}-X_{1,i}\big)
+\mathcal{O}\big(\sqrt[3]{\varepsilon_{\nu}}\big)
= \sum_{i=l}^{m-1} \rho(Y_{i+1})\big(Y_{i+1}-Y_{i}\big) +
\mathcal{O}\big(\sqrt[3]{\varepsilon_{\nu}}\big) \\ 
\approx & \int_{Y_{l}}^{Y_{m}} \rho(x)\,dx \to
\int_{Y_{-}}^{Y_{+}} \rho(x)\,dx\text{ as } \nu \to \infty.
\end{split}\]
We have used that $\rho(x)$ and $u(x)$ together with 
their first derivatives are bounded in order to get 
that $I_{\nu}'=\mathcal{O}(\sqrt[3]{\varepsilon_{\nu}})$. 
Due to the mass conservation 
and the fact that flow maps $[Y_{-},Y_{+}]$ to
$[X_{-},X_{+}]$ we have
\[\mathcal{M}\big([Y_{-},Y_{+}]\big):=\int_{Y_{-}}^{Y_{+}} \rho(x,0)\,dx=
\int_{X_{-}}^{X_{+}}\rho(x,T)\,dx=:\mathcal{M}\big([X_{-},X_{+}]\big),\]
and (\ref{limit_incr}) is proved.
A value of $T_{\mathop{\rm max}}$ is arbitrary here.

Next, suppose that $u(x)$ is decreasing. 
The solution $U^\nu$ consists of shadow waves separating constant states
in the beginning. The first interaction occurs in 
a non-negligible time (see (\ref{dif_decr}) and the analysis there), 
since $Y_{i-1}-Y_i< C\sqrt[3]{\varepsilon_{\nu}}$ for each $i$. 
A classical solution to (\ref{chRS}, \ref{initial_3}) 
with decreasing $u(x)$ exists only until some
time $T_{\mathop{\rm max}}$ when the first pair of characteristics intersect.
That is, shadow waves intersect approximately at the same time as
nearby characteristics.

Let $T<T_{\mathop{\rm max}}$. Take an interval
$[X_{-},X_{+}]$, where $X_{\ast}=Y_{\ast}+u(Y_{\ast})T$, $\ast \in
\{+,-\}$. It is clear that $[Y_{-},Y_{+}]$ maps to $[X_{-},X_{+}]$ when
$t=T$, so $(\ref{limit_incr})_2$ follows.
Suppose that $Y_{-}\in(Y_{l-1},Y_{l}]$, $Y_{+}\in[Y_{m},Y_{m+1})$
for some $l,m \in \mathbb{Z}$ as in the previous case. Denote
$X_{i}:=Y_{i}+y_{i,i+1}T$. Then
\[\begin{split} S_k
:&=\int_{X_k}^{X_{k+1}}\rho^{\nu}(x,T)\,dx=\frac{1}{2}\xi_{k-1,k}T
+\int_{X_k+\frac{\varepsilon_{\nu}}{2}T}^{X_{k+1}
-\frac{\varepsilon_{\nu}}{2}T}
\rho_{k+1}\,dx+\frac{1}{2}\xi_{k,k+1}T\\
&=\frac{1}{2}\big(\xi_{k-1,k}+\xi_{k,k+1}\big)T+\rho_{k+1}\big(X_{k+1}-X_k\big)
-\varepsilon_{\nu} T \rho_{k+1}\\
&=\frac{1}{2}\big(\xi_{k-1,k}+\xi_{k,k+1}\big)T+\rho_{k+1}\big(Y_{k+1}-Y_k\big)
+\rho_{k+1}\big(y_{k,k+1}\!-\!y_{k-1,k}\big)T\!-\!\varepsilon_{\nu} T
\rho_{k+1},
\end{split}\]
since
\[\xi_{k,k+1}T=\sqrt{\rho_k\rho_{k+1}}(u_k-u_{k+1})T =\lim_{\nu\to
\infty}\int_{X_k-\frac{\varepsilon_\nu}{2}T}^{X_k+\frac{\varepsilon_\nu}{2}T}
\rho^{\nu}(x,T)\,dx.\]
Using (\ref{dif_decr}) with $i=k-1$, $j=k+1$ and 
$\rho_{k+1}=\rho_k+\mathcal{O}(\sqrt[3]{\varepsilon_{\nu}})$, 
\begin{equation}\label{y_sim} y_{k,k+1}-y_{k-1,k}= -
\frac{1}{2}\big(u_{k-1}-u_{k+1}\big)
+\mathcal{O}\big(\sqrt[3]{\varepsilon_{\nu}^2}\big).
\end{equation}
Boundedness of $\rho(x)$ implies $\sqrt{\rho_{k}\rho_{k+1}}=
\rho_{k+1}+\frac{1}{2}(\rho_{k}-\rho_{k+1})
+\mathcal{O}\big(\sqrt[3]{\varepsilon_{\nu}^{2}}\big)$, and
$\sqrt{\rho_{k-1}\rho_{k}}=
\rho_{k+1}+\frac{1}{2}(\rho_{k-1}+\rho_{k}-2\rho_{k+1})
+\mathcal{O}\big(\sqrt[3]{\varepsilon_{\nu}^{2}}\big)$.
Together with (\ref{y_sim}), it implies
\begin{equation*}
\begin{split}
\beta_k: = & \frac{1}{2}\big(\xi_{k-1,k}+\xi_{k,k+1}\big)+\rho_{k+1}
\big(y_{k,k+1}-y_{k-1,k}\big) \\ = &
\frac{1}{2}\rho_{k+1}(u_{k-1}-u_{k})
+\frac{1}{2}\rho_{k+1}\big(u_{k}-u_{k+1}\big)
-\frac{1}{2}\rho_{k+1}\big(u_{k-1}-u_{k+1}\big)
+\mathcal{O}\big(\sqrt[3]{\varepsilon_{\nu}^{2}}\big) \\
& +\frac{1}{4}\underbrace{\big((\rho_{k-1}+\rho_{k}-2\rho_{k+1})
(u_{k-1}-u_{k})+ (\rho_{k}-\rho_{k+1}) 
(u_{k}-u_{k+1})\big)}_{\sim \sqrt[3]{\varepsilon_{\nu}^{2}}})
=\mathcal{O}\big(\sqrt[3]{\varepsilon_{\nu}^{2}}\big).
\end{split}
\end{equation*} 
Thus,
\[\begin{split} I_{\nu} := & \int_{X_{-}}^{X_{+}}\rho^{\nu}(x,T)\,dx\\
    = & \sum_{k=l}^{m-1}S_{k}+ \rho_{l}(X_{l}-X_{-})+\frac{1}{2}
\big(\xi_{l,l+1}+\xi_{m,m+1}\big)T+\rho_{m+1}(X_{+}-X_{m})\\ = &
\sum_{k=l}^{m-1} \rho_{k+1}\big(Y_{k+1}-Y_k\big)
+\rho_{l}(Y_{l}-Y_{-})+\rho_{m+1}(Y_{+}-Y_{m})
+T\sum_{k=l}^{m-1}\beta_k \\ &
+\underbrace{\Big(\frac{1}{2}\big(\xi_{l,l+1}+\xi_{m,m+1}\big)
-\rho_l\big(u(Y_{-})-y_{l,l+1}\big)-\rho_{m+1}
\big(y_{m,m+1}-u(Y_{+})\big)\Big)T}_{\sim \sqrt[3]{\varepsilon_{\nu}}} \\ 
& -\varepsilon_{\nu} T \sum_{k=l}^{m-1} \rho_{k+1}, \; \nu \to \infty.
\end{split}\]
The above sum $\sum_{k=l}^{m-1}\beta_{k}$ has $\mathcal{O}(1/\sqrt[3]{\varepsilon_{\nu}})$ globally bounded elements due to the assumption $
\sqrt[3]{\varepsilon_{\nu}}\leq Y_{i}-Y_{i-1}\leq C \sqrt[3]{\varepsilon_{\nu}}$ from Theorem \ref{thm:exist_limit}.
Thus, it is bounded from above by $\mathop{\rm const} \cdot 
\sqrt[3]{\varepsilon_{\nu}}$.  Then
\[\varepsilon_{\nu} T \sum_{k=l}^{r-1} \rho_{k+1}\leq
\sqrt[3]{\varepsilon_{\nu}^{2}}T\sum_{k=l}^{r-1} \rho_{k+1}
\big(Y_{k+1}-Y_k\big)\to 0 \text{ as }\nu \to \infty,\]
since $\rho(x)$ is
bounded. Therefore,
\[I_{\nu}\to \int_{Y_{-}}^{Y_{+}}
\rho(x)\,dx=\mathcal{M}\big([Y_{-},Y_{+}]\big)
=\mathcal{M}\big([X_{-},X_{+}]\big)
\text{ as } \nu \to \infty.\]

The limit $U^{\ast}$ for $t<T_{\max}$ is the weighted
delta measure $\hat{U}_{\ast}$ connecting $(\rho_{0},u_{0})$ and 
the classical solutions obtained by the above procedure.
The life-span $T_{\mathop{\rm max}}$ is determined by the fact that we can
use the above arguments as long as the classical solution exists below $\Gamma$.
That is, as long as characteristics intersect above it. For a neighborhood
of a point $x>R$ their intersection is at the point around $D_{x}$ and 
the assertion follows.

The case when $u(x)$ changes monotonicity finitely many times reduces to
combining these two cases.
\end{proof}

\begin{remark} The life--span $T_{\max}$ equals infinity if $u(x)$ 
  is increasing or if $u'(x)\leq 0$ with small enough absolute value. 
  For a finite $T_{\max}$ we do not know what is distributional limit 
  of solution for $t\geq T_{\max}$, 
  but a solution becomes a single delta shock connecting
  $(\rho(R),u(R))$ and $(\rho(\infty),u(\infty))$ for $t \gg 1$.
\end{remark}

\begin{remark}
Again, the above result is easily extended to system (\ref{PGD})
with the additional energy variable, so all the assertions
in this section hold for that system, too.	
Smooth energy component solves the 
equation $\partial_{t} e + u \partial_{x} e=0$.
\end{remark}

\subsection{Partitions of equidistant type}

Proofs of Theorems \ref{thm:exist_limit} and \ref{existence} are based 
on the compactness argument without any information about a
uniqueness of the limit. We shall now prove that the limit $U^\ast$
given in Theorem \ref{existence} is unique at least for $t<T_{\max}$ 
if partitions of the interval $[R,\infty)$ satisfy the equidistant property: 
Take $\varepsilon$ small enough and define a family of partitions
$\{\mathcal{P}^{\nu}\}_{\nu\in \mathbb{N}_{0}}$ 
in the following way. If $\mathcal{P}^{\nu}=\{Y_i^\nu\}_{i\in \mathbb{N}_{0}}$,
then $\mathcal{P}^{\nu+1} = \mathcal{P}^{\nu} \cup
\big\{Y^\nu_{i+\frac{1}{2}}\big\}_{i\in \mathbb{N}_{0}}$,
where $\sqrt[3]{\varepsilon} \leq  Y^0_{k+1}-Y^0_k 
\leq C \sqrt[3]{\varepsilon}$ 
for each $k$ and some constant $C\geq 1$. If 
$\frac{\sqrt[3]{\varepsilon}}{2^\nu}\leq Y_{k+1}^\nu-Y_{k}^\nu
\leq \frac{C \sqrt[3]{\varepsilon}}{2^{\nu}}=:\mu_{\nu}$ 
for every $k\in\mathbb{N}_0$
and $\nu\in \mathbb{N}_{0}$, the family is said to 
have the equidistant property.
For each partition $\mathcal{P}^{\nu}$ a corresponding  
$U^{\nu}$ is defined in Theorem \ref{thm:exist_limit} 
for $\varepsilon_{\nu}=\varepsilon/2^{3\nu}$.
Denote by $\Gamma^{\nu}: x=c^{\nu}(t)$ the 0-SDW curve in $U^{\nu}$.

\begin{assumption}\label{assump:a} Suppose that $u(x)$ and
$\rho(x)>0$ are continuous and bounded together with their first
derivatives, and $u(x)$ has a finite number of local extremes. 
The values $u_0>\sup_{x\geq R}u(x)$
and $\rho_0>0$ are chosen such that the minimum distance between a
slope of the curve $\Gamma^{\nu}$ and $u(c^{\nu}(t),t)$ is
uniformly greater than zero.
\end{assumption} 

We want to show the uniqueness of the limit $U^{\ast}$
for sequences $\{U^{\nu}\}_{\nu \in \mathbb{N}_0}$ defined by partitions
of equidistant type.
It suffices to show that the
curve $\Gamma$ from Theorem \ref{existence} is unique since
it connects $U_0$ and the unique classical solution $U(x,t)$.

\begin{theorem}\label{thm:uniqueness} If
Assumption \ref{assump:a} holds, then 
a sequence $\{U^{\nu}\}_{\nu \in \mathbb{N}_{0}}$ defined by  
partitions $\{\mathcal{P}^{\nu}\}_{\nu\in \mathbb{N}_{0}}$ 
of equidistant type converges to the unique bounded measure
$U^{\ast}$  in 
$\mathbb{R}\times(0,T_{\max})$ as $\nu\to \infty$.
\end{theorem}

\begin{proof} Let $\varphi\in C_0^\infty\big(\mathbb{R}\times
\mathbb{R}_+\big)$. There exists $\tau_0$,
  $0<\tau_0<T_{\max}$, independent of $\nu \in \mathbb{N}_0$ 
  such that $\varphi$ is
supported by $t>\tau_0$. Our aim is to prove that $\Gamma^{\nu}\to
\Gamma$ as $\nu\to \infty$ in the strip $0<t<T_{\max}$.
Suppose that $c^{\nu}(\tau_0)=c(\tau_0)$ and $\gamma_0\leq \xi^{\nu}(\tau_0)$
independently of a partition. Without loss of generality, assume that
$t=\tau_0$ is interaction time between 0-SDW and a contact discontinuity (or
a shadow wave). That is, for each $\nu$ there exists some $Y_i^{\nu}\in
\mathcal{P}^{\nu}$ such that $Y_i^{\nu}+u(Y_i^{\nu})\tau_0=c^{\nu}(\tau_0)$ 
(or $Y_i^{\nu}+y_{i,i+1}^{\nu}\tau_0=c^{\nu}(\tau_0)$), 
where $y_{i,i+1}^{\nu}$ from 
(\ref{y_{i,k}}) corresponds to the states $u_i^{\nu}$, $u_{i+1}^{\nu}$ in
$\mathcal{P}^{\nu}$. This may not be true in general, 
but a difference would be negligible. 
The compactness of a test function support permits us to take
a points $Y_J^0>R$ and $\overline{T}>0$ as a boundary of the limit analysis. 

For simplicity, we shall suppose first that all partitions $\mathcal{P}^{\nu}$
are equidistant, i.e.\ $Y_{k+1}^{\nu}-Y_k^{\nu}
=\mu_{\nu}$ for each $k\in\mathbb{N}_0$ and $\nu \in \mathbb{N}_0$.
The proof for partitions with equidistant property differs 
only in technical details. Denote 
\[M=\max\{\rho_0,\sup_{x\geq R}\rho(x)\},\, A=
\max\{u_0,\sup_{x\geq R}u(x)\}-\min\{u_0,\inf_{x\geq R}u(x)\}.\]
The following estimates will be used below.
The Taylor expansion formula implies
\begin{equation*}
\begin{split} \xi(t)&=\begin{cases} \gamma+\big(c[\rho]-[\rho
u]\big)t+\frac{\rho_l\rho_r[u]^2-\big(c[\rho]-[\rho
u]\big)^2}{2\gamma}t^2+\mathcal{O}(t^3), &\rho_l\neq \rho_r\\
\gamma+\rho_l(u_l-u_r)t, & \rho_l=\rho_r\neq 0 \end{cases}\\
u_s(t)&=\begin{cases} c+\frac{\rho_l\rho_r[u]^2-\big(c[\rho]-[\rho
u]\big)^2}{\gamma [\rho]}t +\mathcal{O}(t^2), & \rho_l\neq \rho_r\\
c-\frac{2}{\gamma}\rho_l(u_l-u_r)\big(c-\frac{u_r+u_l}{2}\big)
t+\mathcal{O}(t^2), & \rho_l=\rho_r\neq 0, \end{cases}
\end{split}
\end{equation*} for $t$ small enough. Note that the case $\rho_l=\rho_r=0$
is trivial since $\xi'(t)=0,u_s'(t)=0$. The 0-SDW front curve $x=c(t)$,
$c(0)=X$ is approximated by
\[c(t)=\begin{cases} X+ct +
\frac{\rho_l\rho_r[u]^2-\big(c[\rho]-[\rho u]\big)^2}{2\gamma [\rho]}t^2
+\mathcal{O}(t^3), & \rho_l\neq \rho_r\\ X+ct
-\frac{1}{\gamma}\rho_l(u_l-u_r)
\big(c-\frac{u_r+u_l}{2}\big)t^2+\mathcal{O}(t^3),
& \rho_l=\rho_r\neq 0. \end{cases}\]
Also, 
\[\rho_l\rho_r[u]^2-\big(c[\rho]-[\rho
u]\big)^2=-[\rho]^2(c-y_{l,r})(c-z_{l,r}),\]
where $y_{l,r}$ and $z_{l,r}$
are defined in (\ref{y_{i,k}}). The overcompressibility consequences
are the following estimates
\[0<[\rho](c-z_{l,r})\leq
2\max\{\rho_l,\rho_r\}(u_l-u_r)\leq 2MA,\; |c-y_{l,r}|<(u_l-u_r)\leq A,\]
with constants $A$ and $M$ independent of a partition. 
Finally, we have the global estimates
\begin{equation}\label{est}
\big|\xi(t)-\gamma-\big(c[\rho]-[\rho u]\big)t\big|\leq M C_\gamma t^2,\;
\big|c(t)-X-ct\big|\leq C_\gamma t^2,
\end{equation} and
$\big|u_s(t)-c\big|\leq 2C_\gamma t$, where $
C_\gamma:=\frac{MA^2}{\gamma}$.
\medskip

Now, let $u(x)$ be an increasing function.
Take a partition $\mathcal{P}^0=\{Y_k\}_{k\in \mathbb{N}_{0}}$
and its subpartition $\mathcal{P}^1=\mathcal{P}^0
\cup \{Y_{k+\frac{1}{2}}\}_{k\in \mathbb{N}_{0}}$,
where $Y_{k+\frac{1}{2}}=\frac{Y_k+Y_{k+1}}{2}$, $k \in \mathbb{N}_{0}$.
Denote by $(X_{0,j},T_{0,j})$ the point where 0-SDW supported by
$\Gamma^{0}$ meets the contact discontinuity line
$x=Y_j+u_j t$. Denote by $(X_{1,j},T_{1,j})$ 
the intersection point between $\Gamma^{0}$ and the
second contact discontinuity line $x=Y_j+u_{j+1}t$ from $(Y_{j},0)$. 
The intersection points between $\Gamma^1$ and the
first and the second contact discontinuity that originate from $(Y_j,0)$ are
denoted by $(X^1_{0,j},T^1_{0,j})$ and $(X^1_{1,j},T^1_{1,j})$,
respectively. Note that in that case we also have contact discontinuities
originating from the points $(Y_{k+\frac{1}{2}},0)$. 
That produces the new interaction points $(X^1_{m,k+\frac{1}{2}}$, 
$T^1_{m,k+\frac{1}{2}})$, $m=0,1$. Using the above assumptions, we have
$X_{0,i}=X^{\nu}_{0,i}$, $T_{0,i}=T^{\nu}_{0,i}=\tau_0$ 
for each $\nu\in \mathbb{N}_0$. 
Define
\[\begin{split} \gamma_{k,j}&=\xi(T_{k,j}),\; c_{k,j}=u_s(T_{k,j}),\;
X_{k,j}=\gamma(T_{k,j}), \; k=0,1,\; j=i,i+1,\\
\gamma^1_{k,j}&=\xi(T^1_{k,j}),\; c^1_{k,j}=u_s(T^1_{k,j}), \;
X^1_{k,j}=\gamma (T^1_{k,j}), \; k=0,1,\; j=i,i+\tfrac{1}{2},i+1.
\end{split}\]
From
$X_{1,i}= c(T_{1,i})=Y_i+u_{i+1}T_{1,i}$, $c(T_{0,i})=
X_{0,i}=Y_i+u_iT_{0,i}$,
one easily finds
$T_{1,i}=T_{0,i}+\tau_1+\mathcal{O}(\tau_1^2)$, and 
$\tau_1:=\frac{u_{i+1}-u_i}{c_{0,i}-u_{i+1}}T_{0,i}$.
Assumption \ref{assump:a} implies that there exists an $\alpha>0$ such that
$c_{0,i}-u_{i+1}>\alpha$, $i \in \mathbb{N}_{0}$. Then,
$\mathcal{O}\big(\tau_1^2\big)=\mathcal{O}(\mu_{0}^2)$ since
$u_{i+1}-u_i=\mathcal{O}(\mu_{0})$. Note that $\mu_0=\sqrt[3]{\varepsilon}$. 
The estimate  
\begin{equation*}
\big|X_{1,i}-(X_{0,i}+c_{0,i}\tau_1)\big|\leq C_{\gamma_0}\tau_1^2<
C_{\gamma_0}\frac{B_u^2 \bar{T}^2\mu_{0}^2}{\alpha^2},
\end{equation*} with
$B_u:=\sup_{x\geq R}|u'(x)|$ follows from (\ref{est}). 
The new interaction point $(X_{0,i+1},T_{0,i+1})$ is 
a solution to the system of equations
\[X_{0,i+1}=c(T_{0,i+1})=Y_{i+1}+u_{i+1}T_{0,i+1},\;
c(T_{1,i})=X_{1,i}.\]
Thus, $T_{0,i+1}=T_{1,i}+\tau_2+\mathcal{O}\big(\mu_{0}^2\big)$, 
where $\tau_2:=
\frac{Y_{i+1}-Y_i}{c_{1,i}-u_{i+1}}=\frac{\mu_{0}}{c_{1,i}-u_{i+1}}$.
Note that $c_{1,i}-u_{i+1}>c_{0,i}-u_{i+1}>\alpha$ due to the fact
that the speed of shadow wave is increasing in vacuum area, and we have
\begin{equation*}
\big|X_{0,i+1}-(X_{1,i}-c_{1,i}\tau_2)\big|\leq C_{\gamma_0} \tau_2^2<
\frac{C_{\gamma_0}\mu_{0}^2}{\alpha^2}.
\end{equation*}

Let us now consider the partition $\mathcal{P}^1$. Denote 
$\tau^1_1 :=
T_{0,i}\frac{u_{i+\frac{1}{2}}-u_i}{c_{0,i}-u_{i+\frac{1}{2}}}$,
$\tau^1_2 : =\frac{Y_{i+\frac{1}{2}}-Y_{i}}{c^1_{1,i}-u_{i+\frac{1}{2}}}$,
$\tau^1_3:=
T^1_{0,i+\frac{1}{2}}\frac{u_{i+1}-u_{i+\frac{1}{2}}}
{c^1_{0,i+\frac{1}{2}}-u_{i+1}}$,
$\tau^1_4 :=
\frac{Y_{i+1}-Y_{i+\frac{1}{2}}}{c^1_{1,i+\frac{1}{2}}-u_{i+1}}$.
In the same way as for $\mathcal{P}^0$ we have
\[\begin{split} T_{0,i+1}^1
&=T_{0,i}+\tau_1^1+\tau_2^1+\tau_3^1+\tau_4^1
+\mathcal{O}\Big(\frac{\mu_{0}^{2}}{2}\Big),\\
X^1_{0,i+1} &=
\underbrace{X_{0,i}+c_{0,i}\tau^1_1+c^1_{1,i}
\tau^1_2+c^1_{0,i+\frac{1}{2}}\tau^1_3
+c^1_{1,i+\frac{1}{2}}\tau^1_4}_{=:\tilde{X}_{0,i+1}^1}
+\mathcal{O}\Big(\frac{\mu_{0}^{2}}{2}\Big),
\end{split}\]
as well as
\[\begin{split}
|X_{1,i}^1-(X_{0,i}+c_{0,i}\tau_1^1)|,|X_{1,i+\frac{1}{2}}^1
-(X_{0,i+\frac{1}{2}}^1+c_{0,i+\frac{1}{2}}\tau_3^1)|
&<\frac{C_{\gamma_0}}{\alpha^2}B_u^2 \bar{T}^2
\frac{\mu_{0}^{2}}{2}, \\
|X_{0,i+\frac{1}{2}}^1-(X_{1,i}^1+c_{1,i}^1\tau_2^1)|
,|X_{0,i+1}^1-(X_{1,i+\frac{1}{2}}^1+c_{1,i+\frac{1}{2}}^1\tau_4^1)|
&<\frac{C_{\gamma_0}}{\alpha^2}\frac{\mu_{0}^{2}}{2}.
\end{split}\]
There exist positive constants
$C_0$ and $C_1$ such that
\begin{equation}\label{est_tau}
|\tau_1-(\tau_1^1+\tau_3^1)|\leq C_0\frac{\mu_{0}^{2}}{2},\;
|\tau_2-(\tau_2^1+\tau_4^1)|\leq C_1 \frac{\mu_{0}^{2}}{2}.
\end{equation} 
That follows from the estimates
\[\begin{split}
\Big|\frac{1}{c_{0,i}\!-\!u_{i+1}}-\frac{1}{c_{0,i}\!
-\!u_{i+\frac{1}{2}}}\Big|
&<\frac{B_u}{\alpha^2}\frac{\mu_{0}}{2},\\
\Big|\frac{1}{c_{0,i}\!-\!u_{i+1}}-\frac{1}{c_{0,i+\frac{1}{2}}^1\!
-\!u_{i+1}}\Big|& <\frac{C_{\gamma_0}\mu_{0}}{\alpha^2}\big(B_u \bar{T}\!i
+\!1\big). 
\end{split}\]
Thus,
\[|T_{0,i+1}^1-T_{0,i+1}|\leq
(C_0+C_1)\frac{\mu_{0}^{2}}{2}.\]
We have
\begin{equation*}
\begin{split}
\big|X_{0,i+1}-(X_{0,i}\!+\!c_{0,i}\tau_1\!+\!c_{1,i}\tau_2)\big|
&<\frac{C_{\gamma_0}\mu_{0}^{2}}{\alpha^2}\big(B_u^2 \bar{T}^2+1\big),\\
\big|X_{0,i+1}^1-\tilde{X}_{0,i+1}^1\big|&
<2\frac{C_{\gamma_0}}{\alpha^2}\big(B_u^2\bar{T}^2+1\big)
\frac{\mu_{0}^{2}}{2},\\
\big|\tilde{X}_{0,i+1}^1-(X_{0,i}\!+\!c_{0,i}\tau_1\!
+\!c_{1,i}\tau_2)\big|&<C_2\frac{\mu_{0}^{2}}{2}
\end{split}
\end{equation*} 
from (\ref{est_tau}) and the estimates
\[|c_{1,i}-c_{0,i}|<\frac{2C_{\gamma_0}B_u\mu_{0}}{\alpha},
\;|c_{1,i}^1-c_{0,i}|<\frac{2C_{\gamma_0}B_u}{\alpha}\frac{\mu_{0}}{2},
\;|c_{0,i+\frac{1}{2}}^1-c_{1,i}^1|
<\frac{2C_{\gamma_0}}{\alpha}\frac{\mu_{0}}{2}.\]
That proves the existence of the constant $\tilde{C}>0$ such that
\[|X_{0,i+1}^1-X_{0,i+1}|<\tilde{C} \frac{\mu_{0}^{2}}{2}.\]

By repeating the process with each partition $\mathcal{P}^{\nu}$ and
its subpartition $\mathcal{P}^{\nu+1}$, $\nu=1,2,\ldots$, 
we obtain the same estimates with $T_{0,i+1}$,
$T_{0,i+1}^1$ and $\mu_{0}$ substituted by 
$T_{0,i+1}^{\nu}$, $T_{0,i+1}^{\nu+1}$ and
$\mu_\nu$, respectively. Let $T_{0,J}^{\nu}\leq \bar{T}$ be the time
of interaction of $\Gamma^{\nu}$ and the contact discontinuity line 
$x=Y_J+u(Y_J)t$. For each $\nu$
and the partition $\mathcal{P}^{\nu}$ there are at most
$2(Y_J-Y_i)/\mu_{\nu}$ interactions on the compact set. So, we have
\begin{equation*}
\begin{split} \big|T_{0,J}^{\nu}-T_{0,J}^{\nu+1}\big|& \leq
  2(C_0+C_1)(Y_J-Y_i)\frac{\mu_{\nu}}{2}=:C_T\frac{\mu_{\nu}}{2},\\
\big|X_{0,J}^{\nu}-X_{0,J}^{\nu+1}\big|&\leq
2\tilde{C}(Y_J-Y_i)\frac{\mu_{\nu}}{2}=:C_X\frac{\mu_{\nu}}{2}.
\end{split}
\end{equation*} Finally, since $C_X,C_T$ do not depend on partition we
conclude that a distance between the curves $\Gamma^p$ and
$\Gamma^{m+p}$ on $\big(\mathbb{R}\times\mathbb{R}_+\big)\cap
\mathop{\rm supp} \varphi$  can be estimated by
\[\big|X_{0,J}^{m+p}-X^{p}_{0,J}\big|\leq
C_X\sum_{i=p+1}^{m+p}\frac{\mu_{0}}{2^{i}}\leq C_X\frac{\mu_{0}}{2^{p}}
=C_X \mu_{p}, \; 
\big|T_{0,J}^{m+p}-T^{p}_{0,J}\big|\leq C_T\frac{\mu_{0}}{2^{p}}.\]
Thus, $\{\Gamma^{\nu}\}_{\nu\in \mathbb{N}_{0}}$ 
forms a Cauchy sequence, and it converges for
each $t>\tau_0$. To prove the assertion for $t>0$ it is enough to
take $\tau_0$ small enough.
One can prove the assertion in the same way when the function $u(x)$ is
decreasing and $u_0>u(R)$ for $t<T_{\max}$, i.e.\ as long as
characteristics do not intersect below the curve $x=c(t)$. Take the
partition $\mathcal{P}^0$ with $Y_k-Y_{k-1}=\mu_{0}$. 
Suppose that
$\Gamma^0$ meets a shadow wave with a front $x=Y_k+y_{k,k+1}t$ at a
point $(X_k,T_k)$. Assume $X_i=X_i^{\nu}$, $T_i=T_i^{\nu}=\tau_0$ for each
$\nu$. The next interaction point $(X_{i+1},T_{i+1})$ is determined by
\[c(t)=Y_{i+1}+y_{i+1,i+2}t,\; c(T_{i})=X_i,\; u_s(T_i)=c_i.\]

There exists an $\alpha>0$ such that
$c_i-y_{i+1,i+2}>\alpha$
due to Assumption \ref{assump:a}.
Thus, 0-SDW and $\text{SDW}_{i+1,i+2}$ interact at
$t=T_{i+1}$,
\[T_{i+1}=T_i+\tau_i+\mathcal{O}\Big(\frac{\mu_{0}^{2}}{2}\Big),\;
\tau_i:=\frac{1+u'(Y_{i+1})T_i}{c_i-y_{i+1,i+2}}\mu_{0}.\]
That follows from the estimates
$y_{i+1,i+2}-y_{i+1,i}=\frac{1}{2}(u_{i+2}-u_i)
+\mathcal{O}(\mu_{0}^{2})$ and
$u_{i+2}-u_i=2u'(Y_{i+1})\mu_{0}+\mathcal{O}(\mu_{0}^{2})$.
Then
\[
\begin{split}
\big|X_{i+1}-(X_i+c_i\tau_i)\big|&<C_{\gamma_0}\tau_i^2
<C_{\gamma_0}\frac{(1\!+\!B_uT_{\max})^2\mu_{0}^{2}}{\alpha^2},\\
\big|c_{i+1}-c_i\big|&<2C_{\gamma_0}\frac{1\!+\!B_uT_{\max}}{\alpha}\mu_{0}.
\end{split}\]

Take now the subpartition $\mathcal{P}^1$ with
$Y_{i+1}-Y_{i+\frac{1}{2}}=Y_{i+\frac{1}{2}}-Y_{i}=\frac{\mu_{0}}{2}=\mu_{1}$
for each $i \in \mathbb{N}_0$. The
interaction points are denoted by $(X_j^1,T_j^1),j=i,i+\frac{1}{2},i+1$.
Similarly, as in the case of increasing $u(x),$ there exist constants
$D_0,D_1>0$ such that
\[|T^1_{i+1}-T_{i+1}|\leq
D_0\frac{\mu_{0}^{2}}{2},\; |X_{i+1}^1-X_{i+1}|\leq D_1
\frac{\mu_{0}^{2}}{2}.\]

Analogous relations with $\mu_{0}$ replaced by
$\mu_{\nu}$ hold for the partitions $\mathcal{P}^{\nu}$
and $\mathcal{P}^{\nu+1}$.
Let $T_J^{\nu}<T_{\max}$ denotes the last intersection time between
$\Gamma^{\nu}$ and shadow wave in the domain $\mathop{\rm supp}\varphi$.
The error accumulates with each interaction and gives
\[\begin{split}
|X^{\nu}_{J}-X^{\nu+1}_{J}| \leq (Y_J-Y_i)D_1\frac{\mu_{\nu}}{2},\;
|T^{\nu}_{J}-T^{\nu+1}_{J}| \leq (Y_J-Y_i) D_0\frac{\mu_{\nu}}{2}.
\end{split}\]
Hence, one concludes that $\Gamma^{\nu}\to \Gamma$ as $\nu\to \infty$
in the strip $t<T_{\max}$. As $\tau_0$ decreases the first point of curve
$\Gamma$ tends to $(R,0)$. 

In the general case of partition with the equidistant property 
it is easy to prove that the maximum number of
interactions between $\Gamma^0$ and shadow waves equals 
$\frac{2(Y_J-Y_i)}{\sqrt[3]{\varepsilon}}=:\frac{E}{\sqrt[3]{\varepsilon}}$
following the
above procedure for equidistant case. Since the sequence of partitions
$\{\mathcal{P}^{\nu}\}_{\nu\in \mathbb{N}_{0}}$ 
is formed in such a way that each subinterval
$[Y_k^{\nu},Y_{k+1}^{\nu}]$ is divided into two,
not necessary equal parts such that
\[\min\big\{Y_{k+\frac{1}{2}}^{\nu}-Y_{k}^{\nu},Y_{k+1}^{\nu}
-Y_{k+\frac{1}{2}}^{\nu}\big\}\geq
\frac{\mu_{\nu}}{2C},\;
\max\big\{Y_{k+\frac{1}{2}}^{\nu}-Y_{k}^{\nu},Y_{k+1}^{\nu}
-Y_{k+\frac{1}{2}}^{\nu}\big\}\leq
\frac{\mu_{\nu}}{2},\]
one concludes that the number of collisions between
$\Gamma^{\nu}$ and shadow waves is at most $\frac{EC}{\mu_{\nu}}$. Thus, the
above proof holds for general case, with $Y_J-Y_i$ replaced by
$C(Y_J-Y_i)$.
\end{proof}

\end{document}